\documentclass[11pt,english,a4paper]{smfart}
\usepackage{color}

\usepackage[english]{babel}
\numberwithin{equation}{section}

\usepackage{mathabx}
\usepackage{enumerate}
\usepackage{amsmath,amssymb, bm,amsfonts} 
\usepackage{bbm}
\usepackage[all]{xy}
\entrymodifiers={+!!<0pt,\fontdimen22\textfont2>}
\def\apl#1#2#3{#1:
\xymatrix{#2\ar[r]&#3}
}
\usepackage{ mathrsfs }

\newtheorem{theorem}{Theorem}[section]
\newtheorem{lemma}[theorem]{Lemma}

\newtheorem{corollary}[theorem]{Corollary}
\newtheorem{proposition}[theorem]{Proposition}

\newtheorem*{coro*}{Corollary}

{\theoremstyle{definition}}

{\theoremstyle{definition}\newtheorem{example}[theorem]{Example}}

\theoremstyle{definition}
\newtheorem{definition}[theorem]{Definition}
\newtheorem{question}[theorem]{Question}
\newtheorem{fact}[theorem]{Fact}

{\theoremstyle{definition}\newtheorem{remark}[theorem]{Remark}}

\def\T{\ensuremath{\mathbb T}}
\def\R{\ensuremath{\mathbb R}}
\def\Z{\ensuremath{\mathbb Z}}

\def\C{\ensuremath{\mathbb C}}
\def\Q{\ensuremath{\mathbb Q}}
\def\N{\ensuremath{\mathbb N}}

\newcommand{\pss}[2]{\ensuremath{{\langle #1,#2\rangle}}}

\newcommand{\wh}[1]{\widehat{#1}}
\newcommand{\hn}{H_{n}}
\newcommand{\0}{\mkern 1.5 mu\pmb{0}}
\newcommand{\sss}{\mkern 1.5 mu\pmb{s}}
\newcommand{\ttt}{\mkern 1.5 mu\pmb{t}}
\newcommand{\nnn}{\mkern 1.5 mu\pmb{n}}

\newcommand{\tttheta}{\mkern 1.5 mu\pmb{\theta }}
\newcommand{\xxx}{\mkern 1.5 mu\pmb{x}}
\newcommand{\ppp}{\mkern 1.5 mu\pmb{p}}
\newcommand{\qqq}{\mkern 1.5 mu\pmb{q}}
\newcommand{\zzzero}{\mkern 1.5 mu\pmb{0}}
\newcommand{\eeeta}{\mkern 1.5 mu\pmb{\eta}}
\newcommand{\yyy}{\mkern 1.5 mu\pmb{y}}
\newcommand{\eee}{\mkern 1.5 mu\pmb{e}}
\newcommand{\aaaa}{\mkern 1.5 mu\pmb{a}}
\newcommand{\AAA}{\mathcal{A}}
\newcommand{\xn}{x_{n}}
\newcommand{\yn}{y_{n}}
\newcommand{\an}{a_{n}}
\newcommand{\gn}{n\ge 1}
\newcommand{\pxn}{(\xn)_{\gn}}
\newcommand{\pyn}{(\yn)_{\gn}}
\newcommand{\pan}{(\an)_{\gn}}
\newcommand{\oxn}{\oti_{\gn}\xn}
\newcommand{\oyn}{\oti_{\gn}\yn}
\newcommand{\oan}{\oti_{\gn}\an}
\newcommand{\inc}[1]{\bigotimes_{\gn}^{#1}\hn}
\newcommand{\ds}{\displaystyle}
\newcommand{\ba}[1]{\overline{#1}}
\newcommand{\oti}{\otimes}
\newcommand{\opl}{\mathop{\oplus}}
\newcommand{\ti}[1]{\widetilde{#1}}
\newcommand{\jnj}{j\,\in J}
\newcommand{\iij}{i\,\in I_{j}}
\newcommand{\iijj}{_{i,\,j}}
\newcommand{\PPP}{\textbf{\textsf{P}}}
\newcommand{\EE}{\textbf{\textsf{E}}}
\newcommand{\uv}{{u,\,v}}
\newcommand{\ka}{Kazhdan}
\newcommand{\kn}{K_{n}}
\newcommand{\qn}{Q_{n}}
\newcommand{\nc}{n,\,c}
\newcommand{\nw}{n,\,w}

\newcommand{\qq}{Q}

\newcommand{\en}{\varepsilon _{n}}
\newcommand{\jn}{j,\,n}
\newcommand{\ijn}{i,\,j,\,n}
\newcommand{\hh}{\mathscr{H}}
\newcommand{\uu}{\mathscr{U}}

\author{Catalin Badea}
\address{Universit\'e Lille 1, Laboratoire Paul Painlev\'e, CNRS UMR 8524, B\^at. M2,
59655 Villeneuve d'Ascq Cedex, France}
\email{catalin.badea@univ-lille1.fr}
\author{Sophie Grivaux}
\address{CNRS,
Laboratoire Ami\'enois de Math\'{e}matique Fondamentale et Appliqu\'{e}e, UMR 7352,
Universit\'{e} de Picardie Jules Verne,
33 rue Saint-Leu,
80039 Amiens Cedex 1,
France}
\email{sophie.grivaux@u-picardie.fr}
\title[Kazhdan sets and equidistribution properties]{Kazhdan sets in groups and equidistribution properties}

\begin{document}
\begin{abstract}
 Using functional and harmonic analysis methods, we study \ka\ sets in topological groups which do not necessarily have 
Property (T). We provide a new criterion for a generating subset $\qq$ of 
a group $G$ to be a \ka\ set; it relies on the 
existence of a positive number $\varepsilon$ such that every unitary representation of $G$ with a 
$(\qq,\varepsilon )$-invariant vector has a finite dimensional 
subrepresentation. Using this result, we give an equidistribution 
criterion for a generating subset of $G$ to be a \ka\ set. In the 
case 
where $G=\Z$, this shows that if $(n_{k})_{k\ge 1}$ is a sequence of integers such that $(e^{2i\pi \theta 
n_{k}})_{k\ge 1}$ is uniformly distributed in the unit circle for all 
real numbers $\theta $ except at most countably many, then
$\{n_{k}\,;\,k\ge 1\}$ is a \ka\ set in $\Z$ as soon as it generates $\Z$. 
This answers a question of Y.\ Shalom from 
[B.~Bekka, P.~de la~Harpe, A.~Valette, Kazhdan's property (T), 
Cambridge Univ. Press, 2008]. 
We also obtain characterizations of \ka\ sets in second countable locally compact abelian groups, in the Heisenberg groups and in the group $\textrm{Aff}_{+}(\R)$.
This answers in particular a question from 
[B.~Bekka, P.~de la~Harpe, A.~Valette, Kazhdan's property (T), 
op. cit.].
\end{abstract}
\thanks{We would like to thank an anonymous referee for very useful remarks and comments; the present version has been completely rewritten after receiving his/her report. This work was supported in part by the Labex CEMPI (ANR-11-LABX-0007-01) and by the EU IRSES grant AOS (PIRSES-GA-2012-318910).}
\maketitle

\section{Introduction}\label{Section 0}
A \emph{unitary representation} of a topological group $G$ on a Hilbert space $H$ is a group 
morphism 
from 
$G$ into the group $\uu(H)$ of all unitary operators on $H$ which 
is strongly continuous, i.\,e.\ such that the map
$\smash{\xymatrix{g\ar@{|->}[r]&\pi (g)x }}$ is continuous from $G$ into 
$H$ for all vectors $x\in H$. As all the representations  we consider 
in this paper are unitary, we will often drop the word ``unitary'' and 
speak simply of representations of a group $G$ on a Hilbert space $H$. In this paper the Hilbert spaces will always be supposed to be complex, 
and endowed with an inner product $\pss{\,\cdot\,}{\cdot\,}$ 
which is linear in the first variable and antilinear in the second 
variable.
\begin{definition}\label{Definition 1}
 Let $\qq$ be a subset of a topological group $G$, $\varepsilon $ a 
positive 
real number, and $\pi$ a unitary 
representation of $G$ on a Hilbert space $H$. A vector $x\in H$ is said to be 
\emph{$(Q,\varepsilon )$-invariant for $\pi $} if 
\[
\sup_{g\,\in \qq} ||\pi (g)x -x||<\varepsilon ||x ||.
\]
A $(Q,\varepsilon )$-invariant vector for $\pi $ is in particular non-zero.
A \emph{$G$-invariant vector} for $\pi$ is a vector $x\in H$ such that $\pi (g)x=x
$ for all $g\in G$.
\end{definition}
The notions of \ka\ sets and \ka\ pairs will be fundamental in our work.
\begin{definition}\label{Definition 2}
 A subset $\qq$ of a topological group $G$ is a \emph{\ka\ set in $G$} if 
there exists $\varepsilon >0$ such that the following property holds true:
any unitary representation $\pi $ of $G$ on a complex Hilbert space $H$ 
with a 
$(\qq,\varepsilon )$-invariant vector has a non-zero $G$-invariant vector.
\noindent
In this case, the pair $(\qq,\varepsilon )$ is \emph{\ka\ pair}, and 
$\varepsilon $ is a \emph{\ka\ constant for $\qq$}. A group $G$ has \emph{Property} (T), or is a \emph{\ka\ group}, if it 
admits a compact \ka\ set. 
\end{definition}
Property (T)  is a rigidity property of topological groups which has been
introduced by \ka\ in \cite{K} for locally compact groups, and which has 
spectacular applications to many fields. For instance, the groups $SL_{n}(\R)$ and $SL_{n}(\Z)$ have Property (T) if and only if $n\ge 3$. We refer the reader to the monograph \cite{BdHV} by Bekka, de la Harpe, and Valette 
for a comprehensive presentation of Kazhdan's Property (T) and its applications (see also \cite{HarpeVal}).

The aim of this paper is to identify and study Kazhdan sets in topological groups. For discrete groups with Property (T) the Kazhdan sets are known. Recall first the following definition.
\begin{definition}\label{Definition 3}
If $\qq$ is a subset of a group $G$, we denote by 
$\langle\qq\rangle$ the smallest subgroup of $G$ containing $\qq$, i.\,e.\ 
the set of all elements of the form $g_{1}^{\,\pm 1}\dots g_{n}^{\,\pm 
1}$, where $\gn$ and $g_{1},\dots,g_{n}$ belong to $\qq$. We say that 
\emph{$\qq$ generates $G$}, or \emph{is generating in $G$}, if $\langle\qq\rangle=G$. 
\end{definition}
Locally compact groups with Property (T) are compactly generated. In particular, discrete groups with Property (T) are finitely generated 
and it is known (see \cite[Prop.~1.3.2]{BdHV}) that the \ka\ 
 subsets of a discrete group with Property (T) are exactly the  generating subsets of the group. 
More generally \cite[Prop.~1.3.2]{BdHV}, a generating set of a locally compact group which has Property (T) is a Kazhdan set and, conversely, a Kazhdan set which has non-empty interior is necessarily a generating set. 
\par\smallskip
For groups without Property (T) the results about Kazhdan sets and Kazhdan pairs are very sparse. It is known (see \cite[Prop.~1.1.5]{BdHV}) that $(G,\sqrt{2})$ is a \ka\ pair for every topological group $G$, so $G$ is always a (``large'') \ka\ subset of itself.  
The main motivations for the present paper are two questions from \cite[Sec.~7.12]{BdHV}. The first one is due to Y.~Shalom: 

\begin{question}\label{Question 0}\cite[Sec.~7.12]{BdHV}
 ``The question of knowing if a subset $\qq$ of $\Z$ is a \ka\ set is 
possibly related to the equidistribution of the sequence $(e^{2i\pi 
n\theta 
})_{n\,\in Q}$ for $\theta $ irrational, in the sense of Weyl.''
\end{question}

We refer the reader to the classical book \cite{KuiNied} by Kuipers and Niederreiter for more information about  equidistributed (sometimes called uniformly distributed) sequences. Recall that the Weyl Criterion (\cite[Th. 2.1]{KuiNied}) states that if $(x_k)_{k\ge 1}$ is a sequence of real numbers, $(e^{2i\pi x_k})_{k\ge 1}$ is equidistributed in $\T$ if and only if $\frac{1}{N}\sum_{k=1}^N e^{2i\pi h x_k}$ tends to $0$ as $N$ tends to infinity for every non-zero integer $h$. Hence if $(n_k)_{k\ge 1}$ is a sequence of elements of $\Z$, $(e^{2i\pi n_k\theta})_{k\ge 1}$ is equidistributed in $\T$ for every  $\theta\in\R\setminus\Q$ if and only if $\frac{1}{N}\sum_{k=1}^N e^{2i\pi n_k\theta}$ tends to $0$ as $N$ tends to infinity for every $\theta\in\R\setminus\Q$. If $\chi_{\theta}$ denotes, for every $\theta\in\R$, the character on $\Z$ associated to $\theta$, this means that $\frac{1}{N}\sum_{k=1}^N \chi_{\theta}(n_k)$ tends to $0$ as $N$ tends to infinity for every $\theta\in\R\setminus\Q$.
\par\medskip
The first remark about Question \ref{Question 0} is that it concerns Kazhdan \emph{sets} and  equidistributed \emph{sequences}; notice that a rearrangement of the terms of a sequence can destroy its equidistribution properties. 
It is known \cite[p. 135]{KuiNied} that given a sequence of elements of the unit circle $\T$, there exists a certain rearrangement of the terms which is is equidistributed if and only if the original sequence is dense in $\T$. 
The second remark is that, as mentioned before, \ka\ sets of $\Z$ are necessarily generating, while there are non-generating subsets $Q$ of $\Z$, like $Q =  p\Z$ with $p\ge 2$, for which the sequences $(e^{2i\pi pk\theta 
})_{k\,\in \Z}$ are equidistributed for all irrational $\theta$'s. 
So Question 
\ref{Question 0} may be rephrased as follows:
\begin{question}\label{Question 1}
(a) Let $\qq$ be a Kazhdan subset of $\Z$. Does a certain rearrangement $(n_k)_{k\,\ge 1}$ of the elements of $Q$ exist such that $(e^{2i\pi n_{k}\theta })_{k\,\ge 1}$ is equidistributed in 
$\T$ for every $\theta\in\R\setminus\Q$? Equivalently, is the sequence $(e^{2i\pi n\theta })_{n\,\in Q}$ dense in $\T$ for every $\theta\in\R\setminus\Q$?

 (b) Let $\qq=\{n_{k}\,;\,k\ge 1\}$  be a generating subset of $\Z$. Suppose that the 
sequence $(e^{2i\pi n_{k}\theta })_{k\,\ge 1}$ is equidistributed in 
$\T$ for every $\theta\in\R\setminus\Q$. Is $\qq$ a \ka\ set in $\Z$?
\end{question}

We will prove in this paper that Question \ref{Question 1}~(a) has a negative answer, a counterexample being provided by the set $Q = \{2^{k}+k\,;\,k\ge 1\}$ (see Example \ref{Example B}). 
On the other hand, one of the aims of this paper is to show that Question \ref{Question 1}~(b) has a positive answer. Actually, we will consider Question \ref{Question 1}~(b) in the more general framework of Moore groups, and answer it in the affirmative (Theorem \ref{Theorem 3}).
\par\smallskip

The second question of \cite[Sec.~7.12]{BdHV} runs as follows:

\begin{question}\label{Question 00}\cite[Sec.~7.12]{BdHV}
``More generally, what are the Kazhdan subsets 
of  $\Z^k$, $\R^k$, the Heisenberg group, or other infinite amenable groups?''
\end{question}

We shall answer
Question \ref{Question 00} in Section \ref{Section 7} by
giving a complete description of Kazhdan sets in many classic groups which do not have Property (T), including the groups $\Z^k$ and $\R^k$, $k\ge 1$, the Heisenberg groups of all dimensions, and the group $\textrm{Aff}_{+}(\R)$ of orientation-preserving affine homeomorphisms of $\R$.

\section{Main results}\label{Sec2}
Let us now describe our main results in more detail.
\subsection{Equidistributed sets in Moore groups}
In order to state Question \ref{Question 1}~(b) for more general groups, we first need to define equidistributed sequences. There are several possible ways of doing this.
If $(g_k)_{k\ge 1}$ is a sequence of elements in a locally compact group $G$, uniform distribution of $(g_k)_{k\ge 1}$ in any of these senses requires a certain form of convergence, as $N$ tends to infinity, of the means 
\begin{equation}\label{means}
\dfrac{1}{N}\ds\sum_{k=1}^{N}\pi(g_{k})
\end{equation}
to the orthogonal projection $P_{\pi}$ on the subspace of invariant vectors for $\pi$, for a certain class of unitary representations $\pi$ of $G$.
Veech \cite{V1}, \cite{V2} calls $(g_k)_{k\ge 1}$ uniformly distributed in $G$ if the convergence of the means (\ref{means}) holds in the weak operator topology for all unitary representations of $G$ (or, equivalently, for all irreducible unitary representations of $G$, provided $G$ is supposed to be second countable). Unitary uniform distribution in the sense of Losert and Rindler \cite{LR}, \cite{GLR} requires the convergence in the strong operator topology of the means (\ref{means}) for all irreducible unitary representations of the group, while Hartman uniform distribution only requires convergence in the strong operator topology for all finite dimensional unitary representations. 
\par\smallskip
In this paper we deal with the following natural extension to general locally compact groups $G$ of
the equidistribution condition of
 Question \ref{Question 1}~(b):
 if $(g_{k})_{k\ge 1}$ is a sequence of elements of $G$, we require the sequence of means (\ref{means})
to converge to $0$ in the weak topology for all finite dimensional irreducible unitary representations of $G$ except those belonging at most countably many equivalence classes of irreducible representations.
In the case of the group $\Z$, sequences $(n_{k})_{k\ge 1}$ of integers such that $(e^{2i\pi\theta n_k})_{k\ge 1}$ is uniformly distributed in $\T$ for all $\theta\in\R$ except countably many are
said to be of first kind (see for instance \cite{Ha2}). 
The class of groups we will consider in relation to Question \ref{Question 1}~(b) is the class of  second countable Moore groups. Recall that $G$ is said to be a \emph{Moore group} if all irreducible representations of $G$ are finite dimensional. Locally compact Moore groups are completely described in \cite{Mo}: a Lie group is a Moore group if and only 
if it has a closed subgroup $H$ such that $H$ modulo its center is compact, 
and a locally compact group is a Moore group if and only if it is a projective 
limit of Lie groups which are Moore groups. See also the survey \cite{Pal} 
for more information concerning the links between various properties of topological 
groups, among them the property of being a Moore group. Of course 
all locally compact abelian groups are Moore groups.
\par\smallskip
Here is the first main result of this paper.

\begin{theorem}\label{Theorem 3}
 Let $G$ be a second countable locally compact Moore group. 
 Let
$(g_k)_{k\ge1}$ be a sequence of elements of $G$. Suppose that  $(g_{k})_{k\ge 1 } $ satisfies the following equidistribution
assumption:
\begin{equation}\label{Pro3}
 \begin{minipage}{130 mm}
{for all (finite dimensional) irreducible unitary
representations $\pi $ of $G$ on a Hilbert space $H$, except those belonging to at most countably many equivalence classes,}
\[
\newsavebox{\abbaa}
\savebox{\abbaa}{\smash[b]{\xymatrix@C=13pt{\scriptstyle 
N\ar[r]&\scriptstyle+\infty}}}
\xymatrix@C=50pt{\dfrac{1}{N}\ds\sum_{k=1}^{N}\pss{\pi (g_{k})x }{y }\ar[r]_-{\usebox{\abbaa}}&0}
\quad\textrm{for every}\ x,y\in H.
\]
 \end{minipage}
\end{equation}

-- If $Q= \{g_k \,;\,
k\ge 1\}$ generates $G$ (in which case $G$ has to be countable), then $\qq$ is a \ka\ set in $G$.

-- If $\qq$ is not assumed to generate $G$, $\qq$ becomes a \ka\ set when one adds to it a suitable ``small'' perturbation. More precisely, if $(W_{n})_{\gn}$ is an increasing sequence of subsets of $G$ such that $\bigcup_{\gn}W_{n}=G$, there exists $n\ge 1$ such that $W_{n}\cup\qq$ is a \ka\ set in $G$.
\end{theorem}

The equidistribution property (\ref{Pro3}) of the sequence
$(g_{k})_{k\ge 1 }$ takes a more 
familiar form when the group $G$ is supposed to be abelian: it is equivalent to requiring that condition \eqref{(o)} below holds true
for all characters $\chi $ of the group except possibly countably many.

\begin{theorem}\label{Theorem 1}
 Let $G$ be a locally compact abelian group, and let $(g_{k
})_{k\ge 1}$ be  a sequence of elements of $G$. Suppose that 
 \newsavebox{\abba}
\savebox{\abba}{\smash[b]{\xymatrix@C=13pt{\scriptstyle 
N\ar[r]&\scriptstyle+\infty}}}
 \begin{equation}\label{(o)}
 \xymatrix@C=50pt{\dfrac{1}{N}\ds\sum_{k=1}^{N}\chi (g_{k})\ar[r]_-{\usebox{\abba}}&0}
 \end{equation}
for all characters $\chi $ on $G$, except at most countably many. 
If $Q= \{g_k \,;\,
k\ge 1\}$ generates $G$, then $\qq$ is a \ka\ set in $G$.
If $\qq$ is not assumed to generate $G$, and if $(W_{n})_{\gn}$ is an increasing sequence of subsets of $G$ such that $\bigcup_{\gn}W_{n}=G$, then there exists $n\ge 1$ such that $W_{n}\cup\qq$ is a \ka\ set in $G$.
\end{theorem}

Theorem \ref{Theorem 1} can thus be seen as a particular case of Theorem \ref{Theorem 
3}, except for the fact that there is no need to suppose that the group is 
second countable when it is known to be abelian. The case $G=\Z$ provides a positive answer to Question \ref{Question 1}~(b) above.

\subsection{Kazhdan sets and finite dimensional subrepresentations}
The proof of Theorem \ref{Theorem 1} relies on Theorem \ref{Theorem 0} below, which gives a new  
condition for a ``small perturbation'' of a subset $\qq$ of a 
group $G$ to be a \ka\ set in $G$. Theorem \ref{Theorem 0} constitutes the core of the paper, and  has, besides the proofs of Theorems \ref{Theorem 3} and \ref{Theorem 1}, 
several interesting applications
which we will present in Sections \ref{Section 6} and \ref{Section 7}.

\begin{theorem}\label{Theorem 0}
 Let $G$ be a topological group, and let $(W_{n})_{\gn}$ be an increasing sequence of 
subsets of $G$ such that $W_{1}$ is a 
neighborhood of the unit element $e$ of $G$ and $\bigcup_{\gn}W_{n}=G$.
Let $\qq$ be a subset of $G$ satisfying the following assumption:
\begin{equation}\label{Pro1}
 \tag{*}\quad \begin{minipage}{120 mm}
there exists a positive constant $\varepsilon $ such that every unitary 
representation $\pi $ of $G$ on a Hilbert space $H$ admitting a 
$(\qq,\varepsilon )$-invariant vector has a finite dimensional 
subrepresentation.
\end{minipage}
\end{equation}
\par\smallskip 
\noindent Then there exists an integer $\gn$ such that $\qq_{n}=
W_{n}\cup\qq$ is a \ka\ set in $G$.
\par\smallskip 
If the group $G$ is locally compact, the same statement holds true for any increasing
sequence $(W_{n})_{\gn}$ of subsets of $G$ such that 
$\bigcup_{\gn}W_{n}=G$.
\end{theorem}

The condition that $W_{1}$ be a neighborhood of $e$, which appears in the first part of the statement of Theorem \ref{Theorem 0}, will be used in the proof in order to ensure the strong continuity of some infinite tensor product representations (see Proposition \ref{Proposition 3.2}). When $G$ is locally compact, this assumption is no longer necessary (see Proposition \ref{Proposition 3.2.0}).
\par\smallskip
We stress that Theorem \ref{Theorem 0} is valid for all topological groups. We will apply it mainly to groups which do not have Property (T) and to subsets of such groups which are not relatively compact, a notable exception being the proof of Theorem \ref{Theorem 7.1}, where
we retrieve  a characterization of Property (T) 
for $\sigma $-compact locally compact groups due to Bekka and Valette 
\cite{BV}, see also \cite[Th.~2.12.9]{BdHV}. The original proof of this result relies on 
the Delorme-Guichardet theorem that such a group has Property (T) if and 
only if it has property (FH). See Section \ref{Section 6} for more details.
\par\smallskip 
Theorem \ref{Theorem 0} admits a simpler formulation if we build the  
sequence $(W_{n})_{\gn}$ starting from a set which generates the group:

\begin{corollary}\label{Corollary 1}
Let $G$ be a topological group.
 Let $\qq_{0}$ be a subset of $G$ which generates $G$ 
and let
 $\qq$ be a subset of $G$. Suppose either that $\qq_{0}$ has non-empty 
interior, or that $G$ is a locally compact group. If $Q$ satisfies assumption \emph{(\ref{Pro1})} of Theorem 
\ref{Theorem 0}, then $\qq_{0}\cup\qq$ is a \ka\ set in $G$.
\end{corollary}

One of the main consequences of Corollary \ref{Corollary 1} is Theorem 
\ref{Theorem 2} below, which shows in particular that property (\ref{Pro1}) of Theorem 
\ref{Theorem 0} characterizes \ka\ sets among generating sets (and which have non-empty interior -- this assumption has to be added if the group is not supposed to be locally compact).

\begin{theorem}\label{Theorem 2}
Let $G$ be a topological group and let $\qq$ be a subset of $G$ which 
generates $G$. Suppose either that $Q$ has non-empty interior or that $G$ is locally compact. Then the following assertions are equivalent:
\begin{enumerate}[(a)]
\item $\qq$ is a \ka\ set in $G$;\label{aa}

\item there exists a constant $\delta\in (0,1)$ such that every unitary representation $\pi$ of $G$ on a Hilbert space $H$ admitting a vector $x\in H$ such that $\inf_{g\in\qq}\left|\pss{\pi(g)x}{x}\right| > \delta\|x\|^2$ has a finite dimensional subrepresentation;\label{bb}

\item there exists a  constant $\varepsilon>0 $ such that every unitary 
representation $\pi $ of $G$ on a Hilbert space $H$ admitting a 
$(\qq,\varepsilon )$-invariant vector has a finite dimensional 
subrepresentation.\label{cc}
\end{enumerate}
\end{theorem}
The assumption that $\qq$ generates $G$ cannot be dispensed with in 
Theorem \ref{Theorem 2}: $\qq=2\Z$ is a subset of $\Z$ which 
satisfies property (\ref{cc}), but $\qq$ is clearly not a \ka\ 
set in $\Z$. Condition (b) in Theorem \ref{Theorem 2} is easily seen to be equivalent to condition (c), which is nothing else than assumption (\ref{Pro1}) of Theorem \ref{Theorem 0}. Its interest will become clearer in Section \ref{Section 7} below, where it will be used to obtain a characterization of Kazhdan sets in second countable locally compact abelian groups (Theorem \ref{nimporte}). In the case of the group $\Z$, the characterization we obtain (Theorem \ref{nimportebis}) involves a classic class of sets in harmonic analysis, called \emph{Kaufman sets}. We give in Section \ref{Section 7}
several examples of ``small'' Kazhdan sets in $\Z$, 
describe \ka\ sets in the Heisenberg groups $H_{n}$, $n\ge  1$ (Theorem \ref{Proposition F}), and also in the  group $\textrm{Aff}_{+}(\R)$ (Theorem \ref{Proposition N}). These results provide an answer to Question \ref{Question 00}.
\par\smallskip
The paper also contains an appendix which reviews some constructions of infinite tensor product representations on Hilbert spaces,  used in the proof of Theorem \ref{Theorem 0}.

\section{Mixing properties for unitary representations 
and an abstract version of the Wiener theorem}\label{newSectionW}
\subsection{Ergodic and mixing properties for unitary representations}\label{Section 3}
We 
first recall in this section some definitions and results concerning the structure of unitary representations of a topological group $G$. They can be found for instance in the book \cite{Ke}, the notes 
\cite{Pet}, and the paper \cite{BR} by Bergelson and
Rosenblatt.
\par\smallskip 
Recall that the class \mbox{WAP$(G)$} of \emph{weakly almost 
periodic 
functions} on $G$ is defined as follows: if 
$\ell^{\,\infty }(G)$ denotes the space of bounded functions on $G$,
$f\in\ell^{\,\infty }(G)$ belongs 
to \mbox{WAP$(G)$} if the weak closure in $\ell^{\,\infty }(G)$ of the 
set $\{f(s^{-1}\centerdot)\,;\,s\in G\}$ is weakly compact. For each $s\in G$,
$f(s^{-1}\centerdot)$ denotes the function 
$\smash{\xymatrix{t\ar@{|->}[r]&f(s^{-1}t)}}$ on $G$. By comparison, recall that $f\in\ell^{\,\infty }(G)$ is an \emph{almost periodic function} on $G$, written $f\in\mbox{AP$(G)$}$, if the norm closure in $\ell^{\,\infty }(G)$ of $\{f(s^{-1}\centerdot)\,;\,s\in G\}$ is compact.
If $\pi $ is a 
unitary representation of $G$ on a Hilbert space $H$, the functions
\[
\pss{\pi (\centerdot)x}{y},\quad \bigl|\pss{\pi (\centerdot)x
}{y } \bigr|,\quad  \textrm{and}\quad \bigl|\pss{\pi (\centerdot)x
}{y } \bigr|^{2}\!,
\]
where $x $ and $y $ are any vectors of $H$, belong to 
\mbox{WAP$(G)$}. 
For more on weakly almost periodic functions on a group, see for instance \cite{Bur} or \cite[Ch. 1, Sec. 9]{Gl}. 
The interest of the class of weakly almost periodic functions on $G$ in 
our context is that there exists on \mbox{WAP$(G)$} a unique $G$-invariant 
mean $m$. It satisfies
\[
m(f(s^{-1}\centerdot))=m(f(\,\centerdot\,\, s^{-1}))=m(f)
\]
for every $f\in\mbox{WAP$(G)$}$ and every $s\in G$.
The abstract ergodic theorem then states that if 
$\pi $ is a unitary representation 
of $G$ on $H$\!, 
$
m(\pss{\pi (\centerdot)x }{y })=\pss{P_{\pi }x }{y }
$
for every vectors $x,y\in H$, where $P_{\pi }$ denotes the 
projection of $H$ onto the space
$
E_{\pi }=\{x  \in H\,;\,\pi (g)x =x \ \ \textrm{for every}\ 
g\in G\}
$
of $G$-invariant vectors for $\pi $. The representation $\pi $ is \emph{ergodic} 
(i.\,e.\ 
admits no non-zero $G$-invariant vector) if and only if $m(\pss{\pi 
(\centerdot)x }{y})=0$ for every $x,y \in H$.
\par\smallskip 
Following \cite{BR}, let us now recall that the representation 
$\pi $ is said to be \emph{weakly mixing}
if $m(|\pss{\pi (\centerdot)x }{x }|)=0$ for every $x\in H$, or, equivalently, 
$m(|\pss{\pi (\centerdot)x}{x }|^{2})=0$ for every $x\in H$. Then 
$m(|\pss{\pi (\centerdot)x }{y }|)=m(|\pss{\pi (\centerdot)x}{y
}|^{2})=0$ for every $x,y\in H$.
\par\smallskip 
We will need the following characterization of weakly mixing representations. 
\begin{proposition}\label{Proposition 4.1}
 Let $\pi $ be a unitary representation of $G$ on a Hilbert space $H$. 
The following assertions are equivalent:
\begin{enumerate}
 \item [(1)]$\pi $ is weakly mixing;
 \item[(2)]$\pi $ admits no finite dimensional subrepresentation;
 \item[(3)]$\pi \oti\ba{\pi } $ has no non-zero $G$-invariant vector.
\end{enumerate}
\end{proposition}

Here $\ba{\pi }$ is the conjugate representation of $\pi$. 
The representation $\pi \oti\ba{\pi }$ is equivalent to a 
representation of $G$ on the space $HS(H)$ of Hilbert-Schmidt 
operators on $H$, which is often more convenient to work with. Recall that 
$HS(H)$ is a Hilbert space when endowed with the scalar product defined 
by the formula $\pss{A}{B}=\textrm{tr}(B^{*}\!A)$ for every $A,B\in 
HS(H)$. The space $H\oti \ba{H}$, where $\ba{H}$ is the conjugate of $H$, is identified to $HS(H)$ by 
associating to each elementary tensor product $x\oti\ba{y }$ of 
$H\oti\ba{H}$ the rank-one operator $\pss{\,\centerdot\,}{y}\,{x
}$ on $H$. 
This map $\smash{\apl{\Theta }{H\oti\ba{H}}{HS(H)}}$ extends into a unitary 
isomorphism, and we have for every $g\in G$ and every $T\in HS(H)$
\[
\Theta\, \pi \oti\ba{\pi } (g)\,\Theta ^{-1}(T)=\pi (g)\,T\,\pi (g^{-1}).
\]
We will, when needed, identify $\pi \oti\ba{\pi }$ with this equivalent 
representation, and use it in particular in Section \ref{Section 
4} to obtain a 
concrete description of the space $E_{\pi \oti\ba{\pi }}$ of $G$-invariant 
vectors for $\pi \oti\ba{\pi }$, which is identified to the subspace of 
$HS(H)$
\[
\EE_{\pi}=\{T\in HS(H)\,;\,\pi (g)\,T=T\,\pi (g)\  
\textrm{for every}\ g\in G\}\cdot 
\]

\subsection{Compact unitary representations} \label{Section 4.1}
A companion to the property of weak mixing for unitary representation is 
that of compactness: given a unitary 
representation $\pi $ of $G$ on $H$ a vector $x\in 
H$ is \emph{compact} for $\pi $ if the norm closure of the set $\{\pi 
(g)x\,;\,g\in G\}$ is compact in $H$. The representation $\pi $ itself 
is said to be \emph{compact} if every vector of $H$ is compact for $\pi 
$. Compact representations decompose as direct sums of irreducible 
finite dimensional representations. 
The general structural result for unitary representations is given 
by the following result.

\begin{proposition}\label{Proposition 4.2}
 A unitary representation $\pi $ of $G$ on a Hilbert space $H$ 
decomposes as a direct sum of a weakly mixing representation and a compact 
representation: 
\[
 \smash[b]{H=H_{w}\opl\limits^{\perp}H_{c}\,,}
\]
 where 
$H_{w}$ and $H_{c}$ are both $G$-invariant closed subspaces of $H$, 
$\pi_{w}=\pi \vert_{H_{w}}$ is weakly mixing and $\pi_{c}=\pi 
\vert_{H_{c}}$ is compact. Hence $\pi $ decomposes as a direct sum of a 
weakly mixing representation and finite dimensional 
irreducible subrepresentations.
\end{proposition}
See \cite[Ch.~1]{Pet}, \cite[Appendix~M]{BdHV}, \cite{BR} or \cite{Dye} 
(in the amenable case) for detailed proofs of these results.

\par\smallskip

Now let $\pi $ be a compact representation of $G$ on a Hilbert 
space $H$,
decomposed as a direct sum of irreducible finite dimensional 
representations of $G$. We sort out these representations by equivalence 
classes, and index the distinct equivalence classes by an index $j$ 
belonging to a set $J$, which may be finite or infinite (and which is 
countable if $H$ is separable). For every $\jnj$, we index 
by $\iij$ all the 
representations appearing in the decomposition of $\pi $ which are in the 
$j$-th equivalence class. More precisely, we can decompose $H$ and $\pi 
$ as 
\[
H=\opl\limits_{\jnj}\bigl(\opl\limits_{\iij}H\iijj 
\bigr)\quad 
\textrm{and}\quad  \pi 
=\opl\limits_{\jnj}\bigl(\opl\limits_{\iij}\pi \iijj 
\bigr)
\]
respectively, where the following holds true: 
\begin{enumerate}
 \item [--] for every $\jnj$, the spaces $H\iijj$, $i\in I_{j}$, are equal. 
We denote by 
$K_{j}$ this common space, and by $d_{j}$ its dimension (which is finite).
We also write 
\[
\ti{H}_{j}=\opl\limits_{\iij}H\iijj,\quad\textrm{so that}\quad 
H=\opl_{\jnj}\ti{H}_{j};
\]
\item[--] for every $\jnj$, there exists an irreducible representation 
$\pi 
_{j} $ of $G$ on $K_{j}$ such that $\pi \iijj$ is equivalent to $\pi _{j}$ 
for every $\iij$;
\item[--] if $j$, $j'$ belong to $J$ and $j\neq j'$, $\pi _{j}$ and 
$\pi_{j'}$ are not equivalent.
\end{enumerate}
\par 
Without loss of generality, we will suppose that $\pi \iijj=\pi _{j}$ for every $\iij$. However, we will keep the notation $H\iijj$ for the various orthogonal copies of  the space $K_j$ which appear in the decomposition of $H$, as discarding this notation may be misleading in some of the proofs presented below.
\par\smallskip 
Let $A\in \mathscr{B}(H)$. We write $A$ in block-matrix form with respect to the decompositions
\[
 H=\opl\limits_{\jnj}\bigl(\opl\limits_{\iij}H\iijj
\bigr)
\quad \textrm{ and }\quad
H=\opl\limits_{\jnj}\ti{H}_{j}
\]
as
\[
A=\bigl(A_{\uv}  \bigr)_{k,\,l\,\in\, J,\, u\,\in\, 
I_{k},\, v\,\in\, I_{l}}
\quad \textrm{ and }\quad
A=\bigl(
\ti{A}_{k,\,l}\bigr)_{k,\,l\,\in  J}
\quad
\textrm{respectively}
\cdot 
\]
For every $j\in J$ and every $u,v\in 
I_{j}$, we denote by $i_{\uv}^{\,(j)} $ the identity operator from 
$H_{u,\,j}$ into $H_{v,\,j}$. 

\subsection{A formula for the projection $\PPP_{\pi }$ of $HS(H)$ on 
$\EE_{\pi }$.}\label{Section 4}
We now give an explicit formula for the 
projection $\PPP_{\pi }A$ of a Hilbert-Schmidt operator $A\in HS(H)$ on the following closed 
subspace of $HS(H)$:
\[
\EE_{\pi }=\{T\in HS(H)\,;\,\pi (g)\,T=T\,\pi (g)\ \textrm{for 
every}\ g
\in G\}.
\]
We also compute the norm of $\PPP_{\pi }A$.
\begin{proposition}\label{Proposition 5.1}
 Let $\pi $ be a compact representation of $G$ on $H$, written in the form $\pi 
=\opl\limits_{\jnj}\bigl(\opl\limits_{\iij}\pi_{j}
\bigr)$ as discussed in Section \ref{Section 4.1} above. For every 
 operator $A\in HS(H)$, 
we have
\[
 \PPP_{\pi }A=\sum_{\jnj}\,\dfrac{1}{d_{j}}\sum_{u,\,v\,\in\, 
I_{j}}\,\emph{tr}\bigl(A_{u,\,v}^{} 
\bigr)\,i_{\uv}^{\,(j)}
\quad\textrm{ and }\quad
||\PPP_{\pi }A||^{2}=\sum_{\jnj}\,\dfrac{1}{d_{j}}\sum_{u,\,v\,\in 
I_{j}}\bigl|\emph{tr}
\bigl(A_{u,\,v}^{} 
\bigr)\bigr|^{2}.
\]
\end{proposition}
\par\smallskip
The proof of Proposition \ref{Proposition 5.1} relies on the following straightforward
lemma:
\begin{lemma}\label{Lemma 5.3}
 The space $\EE_{\pi }$ consists of the operators $T\in\emph{HS}(H)$ such 
that 
\begin{enumerate}
 \item [--] for every $k,l\in J$ with $k\neq l$, 
$\ti{T}_{k,\,l}=0$;
\item[--] for every $k\in J$ and every $u,v\in I_{k}$, there exists a 
complex number $\lambda _{\uv} $ such that $T_{\uv} =\lambda _{\uv} 
\,i_{\uv}^{\,(k)}$. Thus $\ti{T}_{k,\,k}=\bigl(\lambda _{\uv} 
\,
i_{\uv}^{\,(k)}
\bigr)_{u,\,v\,\in I_{k}}$.
\end{enumerate}
\end{lemma}
\begin{proof}[Proof of Lemma \ref{Lemma 5.3}]
Let $T\in \EE_{\pi }$. For every $k,l\in J$, $u\in I_{k}$ and 
$v\in I
_{l}$,
$\pi 
_{k}(g)\,T_{\uv} =T_{\uv} \,\pi _{l}(g)$ for every $g\in G$.
Thus the operator 
${T_{\uv}}$ intertwines the two representations $\pi _{k}$ and $\pi _{l}$. 
If $T_{\uv} $ is non-zero, it follows from Schur's Lemma 
that $T_{\uv} $ is an isomorphism. The representations $\pi _{k}$ and
$\pi _{l}$ are thus isomorphically (and hence unitarily) equivalent. Since 
$\pi _{k}$ and $\pi _{l}$ are not equivalent for $k\neq l$, it follows 
that $T_{\uv} =0$ in this case.
If now $k=l$, Schur's Lemma again implies that $T_{\uv}=\lambda
_{\uv}\,i_{\uv}^{\,(k)}$ for some scalar $\lambda _{\uv}$. Thus any operator $T\in
\EE_{\pi }$ satisfies the two conditions of the lemma. The converse is 
obvious.
\end{proof}

The proof of Proposition \ref{Proposition 5.1} is now easy.

\begin{proof}[Proof of Proposition \ref{Proposition 5.1}]
Consider, for every $\jnj$ and $u,v\in I_{j}$, the 
one-dimensional 
subspace $\EE_{\uv}^{\,(j)}$ of $HS(H)$ spanned by the operator
$i_{\uv}^{\,(j)}$. 
These subspaces are pairwise orthogonal in \mbox{HS$(H)$}, and by Lemma 
\ref{Lemma 5.3} we have
\[
\EE_{\pi }=\opl\limits_{\jnj}\,\bigl(\opl\limits_{\uv\,\in I_{j}}\EE_{\uv}^{\,(j)} 
\bigr).
\]
Hence, for every $A\in HS(H)$,
\[
 \PPP_{\pi }A=\sum_{\jnj}\,\sum_{\uv\,\in I_{j}}\Bigl\langle 
A,\dfrac{i_{\uv}^{\,(j)}}{\,\,||i_{\uv}^{\,(j)}||_{HS}}\Bigr\rangle 
\,\dfrac{i_{\uv}^{\,(j)}}{\,\,||i_{\uv}^{\,(j)}||_{HS}}=\sum_{\jnj}\,\,\dfrac{1}{d_{j}}\sum_{\uv\,
\in I_{j}}\textrm{tr}
\bigl(A_{u,\,v}^{} 
\bigr)\,i_{\uv}^{\,(j)},
\]
which gives the two formulas we were looking for. 
\end{proof}
\par\smallskip
\begin{corollary}\label{Corollary 5.4}
 Let $\pi 
=\opl\limits_{\jnj}\bigl(\opl\limits_{\iij}\pi_{j}
\bigr)$ be a compact representation of $G$ on $H$. Let
$x 
=\opl\limits_{\jnj}\bigl(\opl\limits_{\iij}x \iijj 
\bigr)$ and $y 
=\opl\limits_{\jnj}\bigl(\opl\limits_{\iij}y \iijj 
\bigr)$ be two vectors of $H$, and let $A\in\mbox{HS$(H)$}$ be  the 
rank-one 
operator $\pss{\,\centerdot\,}{y}\,x$. Then
\[
\PPP_{\pi }A=\sum_{\jnj}\,\,\dfrac{1}{d_{j}}\sum_{\uv\,\in I_{j}}
\pss{x_{u,\,j}}{y_{v,\,j}}\,i_{\uv}^{\,(j)}.
\]
\end{corollary}

\begin{proof}
 For every $\jnj$ and $u,v\in I_{j}$, 
$A_{\uv}=\pss{\,\centerdot\,}{y_{v,\,j}}\,x_{u,\,j}$, so that 
$
\textrm{tr}\bigl(A_{\uv} \bigr)=
\pss{x_{u,\,j}}{y_{v,\,j}}.
$
The result then follows from Proposition \ref{Proposition 5.1}.
\end{proof}
\subsection{An abstract version of the Wiener Theorem.}\label{Section 4.3}
As recalled in Section \ref{Section 3}, $\EE_{\pi }$ is the space 
of $G$-invariant vectors for the representation $\pi \oti\ba{\pi }$ on 
$HS(H)$, where for every $x,y\in H$, $x\oti\ba{y}$ is identified with the rank-one 
operator $\pss{\,\centerdot\,}{y}\,x$. For every pair 
$(x,y)$ of vectors of $H$, denote by $\pmb{b}_{x,\,y}$ the element of 
$K\oti K$, with
$K=\opl\limits_{\jnj}K_{j}$, defined by 
\[
\pmb{b}_{x,\,y}=\sum_{\jnj}\,\dfrac{1}{\sqrt{d_{j}}}\,\Bigl(\,\sum_{\iij} 
x\iijj\oti\ba{y}\iijj\Bigr).
\]
It should be pointed out that for a fixed index $j\in J$ the vectors 
$x\iijj$ and $y\iijj$ are 
understood in the formula above as belonging to the same space $K_{j}$ 
(and not to the various orthogonal spaces $H\iijj$). So $\pmb{b}_{x,\,y}$ 
is a 
vector of $K\oti K$, not of $H\oti H$. Thus
\[
||\pmb{b}_{x,\,y}||^{2}=\sum_{\jnj}\,\dfrac{1}{{d_{j}}} \sum_{\uv\,\in I_{j}}
\pss{x_{u,\,j}}{x_{v,\,j}}\,\ba{\pss{y_{u,\,j}}{y_{v,\,j}}}.
\]
Combining Corollary \ref{Corollary 5.4} with the formula
\[
m\bigl(|\pss{\pi (\,\centerdot\,)x}{y}|^{2} \bigr)=\pss{P_{\pi
\oti\,\ba{\vphantom{t}\pi }}\,\,x\oti\ba{x}}{y\oti\ba{y}\,}=\pss{\PPP_{\pi 
}\pss{\,\centerdot\,}{x}x}{\pss
{\,\centerdot\,}{y}{y}}
\]
yields
\begin{corollary}\label{Corollary 5.5}
 Let $\pi 
=\opl\limits_{\jnj}\bigl(\opl\limits_{\iij}\pi_{j}
\bigr)$ be a compact representation of $G$ on $H$.
For every vectors 
$x 
=\opl\limits_{\jnj}\bigl(\opl\limits_{\iij}x \iijj 
\bigr)$ and $y 
=\opl\limits_{\jnj}\bigl(\opl\limits_{\iij}y \iijj 
\bigr)$ of $H$, we have
\begin{align}\label{Eq3bis}
  m\bigl(|\pss{\pi (\,\centerdot\,)x}{y}|^{2} 
\bigr)&=\sum_{\jnj} \dfrac{1}{{d_{j}}} \sum_{\uv\,
 \in I_{j}}\pss{x_{u,\,j}}{x_{v,\,j}}\,.\,
 \ba{\pss{y_{u,\,j}}{y_{v,\,j}}}
 =||\pmb{b}_{\,x,\,y}||^{2}.
\end{align}
\end{corollary}
We thus obtain the following abstract version of the Wiener Theorem for 
unitary representations of a group $G$:
\begin{theorem}\label{Theorem 5.6}
 Let $\pi =\pi _{w}\opl \pi _{c}$ be a unitary representation of $G$ on a Hilbert 
space $H=H_{w}\opl H_{c}$, where 
$\pi_{w}$ is the weakly mixing part of $\pi $ and $\pi 
_{c}$ its compact part. Writing $\pi 
_{c}=\opl\limits_{\jnj}\bigl(\opl\limits_{\iij}\pi_{j}
\bigr)$ as  above, we have for every 
vectors $x=x_{w}\opl x_{c}$ and $y=y_{w}\opl y_{c}$ of $H$
\begin{equation}\label{Eq4}
 m\bigl(|\pss{\pi (\,\centerdot\,)x}{y}|^{2} 
\bigr)=||\,\pmb{b}_{\,x_{c},\,y_{c}}||^{2}.
\end{equation}
\end{theorem}
\begin{proof}
 As we have
$
m(|\pss{\pi (\,\centerdot\,)x}{y}|^{2})=
m\bigl(|\pss{\pi_{w} (\,\centerdot\,)x_{w}}{y_{w}}|^{2} 
\bigr)+m\bigl(|\pss{\pi_{c} (\,\centerdot\,)x_{c}}{y_{c}}|^{2} 
\bigr)
$ and
$m(|\pss{\pi_{w} (\,\centerdot\,)x_{w}}{y_{w}}|^{2})=0$,
this follows from Corollary \ref{Corollary 5.5}.
\end{proof}
\par 
We finally derive an
inequality on the quantities $m\bigl(|\pss{\pi (\,\centerdot\,)x}{y}|^{2} 
\bigr)$ for a compact representation $\pi$, which is a direct consequence of Corollary \ref{Corollary 5.5}.
This inequality will be a crucial tool for the proof of our main 
result, to be given in Section \ref{Section 5}.
Using the same notation as in the statement of Corollary \ref{Corollary 5.5}, we
denote by $x=\opl_{\jnj}\ti{x}_{j}$ and $y=\opl_{\jnj}\ti{y}_{j}$ the 
respective decompositions of the vectors $x$ and $y$ of $H$ with respect 
to the decomposition $H=\opl_{\jnj}\ti{H}_{j}$ of $H$. 
Applying the Cauchy-Schwarz inequality twice to (\ref{Eq3bis}) yields the following inequalities:

\begin{corollary}\label{Corollary 5.7}
Let $\pi$ be a compact representation of $G$ on $H$.
 For every 
vectors $x$ and $y$ of $H$, we have
\[
m\bigl(|\pss{\pi (\,\centerdot\,)x}{y}|^{2} 
\bigr)\le\sum_{\jnj}\,\dfrac{1}{{d_{j}}}\, ||\ti{x}_{j}||^{2}.\,||\ti{y}_{j}||^{2}\le\sum_{\jnj}\, ||\ti{x}_{j}||^{2}.\,||\ti{y}_{j}||^{2}.
\]
\end{corollary}

\subsection{Why is (\ref{Eq4}) an abstract version of the Wiener 
Theorem?}\label{Section 4.4}
Theorem \ref{Theorem 5.6} admits a 
much simpler formulation in the case where $G$ is an abelian group.
If $\pi$ is a compact representation of $G$, the formula 
(\ref{Eq3bis}) becomes 
\[
m\bigl(|\pss{\pi (\,\centerdot\,)x}{y}|^{2} 
\bigr)=\sum_{\jnj}\
\sum_{\uv\,\in\,I_{j}}x_{u,\,j}\,\ba{x}_{v,\,j}\,\ba{y}_{u,\,j}
y_{v,\,j}
\]
where $x_{i,\,j}$ and $y_{i,\,j}$, $i\in I_{j}$, $\jnj$, are simply 
scalars.
Using the notation of Corollary \ref{Corollary 5.7}, we have
\begin{equation}\label{Eq6}
m\bigl(|\pss{\pi (\,\centerdot\,)x}{y}|^{2} 
\bigr)=\sum_{\jnj}\,\Bigl|\sum_{u\,\in\,I_{j}}x_{u,\,j}\,\ba{y}_{u,\,j}
\Bigr|^{2}= 
\sum_{\jnj}\,\bigl|\pss{\ti{x}_{j}}{\ti{y}_{j}}\bigr|^{2}\!.
\end{equation}
For every character $\chi \in\Gamma $ (where $\Gamma $ denotes the 
dual 
group of $G$), we denote by $E_{\chi }$ the subspace of $H$
\[
E_{\chi }=\{x\in H\,;\,\pi (g)x=\chi (g)x\ \textrm{for every}\ g\in G\}
\]
and by $P_{\chi }$ 
the orthogonal projection of $H$ on $E_{\chi}$.
Each representation $\pi _{j}$, $j\in J$, being in fact a character $\chi 
_{j}$ on the group $G$, we can identify the space $\ti{H}_{j}$ with 
$E_{\chi_{j} }$. Equation (\ref{Eq6}) then yields the following corollary:

\begin{corollary}\label{cor+}
 Let $G$ be an abelian group, and let $\pi$ be a representation of $G$ on a Hilbert space $H$. Then we have for every $x,y\in H$
 \[
m\bigl(|\pss{\pi (\,\centerdot\,)x}{y}|^{2} 
\bigr)=\sum_{\jnj}\,\bigl|\pss{P_{E_{\chi _{j}}}x}{P_{E_{\chi _{j}}}y}
\bigr|^{2}= 
\sum_{\chi\,\in\,\Gamma }\,\bigl|\pss{P_{E_{\chi}}x}{P_{E_{\chi 
}}y}\bigr|^{2} \!.
\]
In particular, if $x=y$,
\[
m\bigl(|\pss{\pi (\,\centerdot\,)x}{x}|^{2} 
\bigr)=\sum_{\chi\,\in\,\Gamma }\,\bigl|\bigl|P_{E_{\chi 
}}x\bigr|\bigr|^{4}\!.
\]
\end{corollary}

Specializing Corollary \ref{cor+} to the case where $G=\Z$ yields that for any 
unitary operator $U$ on $H$ and any vectors $x,y\in H$,
\newsavebox{\zzz}
\savebox{\zzz}{\smash[b]{\xymatrix@C=13pt{\scriptstyle 
N\ar[r]&\scriptstyle+\infty}}}
\[
\xymatrix@C=60pt{
\dfrac{1}{2N+1}\ds\sum_{n=-N}^{N}\,\bigl|\pss{U^{n}x}{y}\bigr|
^{2}\ar[r]\ar@{}@<.6ex>[r]_-{\usebox{\zzz}}&\ds\sum_{\lambda 
\,\in\,\T}\,\bigl|\pss{P_{\ker(U-\lambda 
\textrm{Id}_{H})}x}{P_{\ker(U-\lambda 
\textrm{Id}_{H})}y}\bigr|^{2}\!.
}
\]
In particular, we have
\newsavebox{\ccc}
\savebox{\ccc}{\smash[b]{\xymatrix@C=13pt{\scriptstyle 
N\ar[r]&\scriptstyle+\infty}}}
\begin{equation}\label{eq:wo}
\xymatrix@C=60pt{
\dfrac{1}{2N+1}\ds\sum_{n=-N}^{N}\,\bigl|\pss{U^{n}x}{x}\bigr|
^{2}\ar[r]\ar@{}@<.6ex>[r]_-{\usebox{\ccc}}&\ds\sum_{\lambda 
\,\in\,\T}\,\bigl|\bigl|P_{\ker(U-\lambda 
\textrm{Id}_{H})}x\bigr|\bigr|^{4}\!.
}
\end{equation}
If $\sigma$ is a probability measure on the unit 
circle $\T$, the operator $M_{\sigma }$ of multiplication by $e^{i\theta }$ on $L^{2}
(\T,\sigma )$ is unitary. Applying \eqref{eq:wo} to $U=M_{\sigma}$ and to $x = 1$, the constant function equal to $1$, we obtain Wiener's Theorem:
\newsavebox{\abb}
\savebox{\abb}{\smash[b]{\xymatrix@C=13pt{\scriptstyle 
N\ar[r]&\scriptstyle+\infty}}}
\begin{equation}\label{Eq5}
 \xymatrix@C=60pt{\dfrac{1}{2N+1}\ds\sum_{n=-N}^{N}|\,\widehat{\sigma 
}(n)|^{2}\ar[r]\ar@{}@<.6ex>[r]_-{\usebox{\abb}}&\ds\sum_{
\lambda\,\in\,\T}\,\sigma (\{\lambda \})^{2}}.
\end{equation}

We refer the reader to \cite{AnBi,Ballo,BalloGold,BjFi,Farkas} and the references therein for  related 
aspects and generalizations of Wiener's theorem. 

\par\medskip
We now have all the necessary tools for the proof of Theorem \ref{Theorem 
0}, which we present in the next section.

\section{Proof of Theorem \ref{Theorem 0}}\label{Section 5}
\subsection{Notation}\label{Section 5.1}
Let $(W_{n})_{\gn}$ be an increasing sequence of subsets of $G$ satisfying the 
assumptions of Theorem \ref{Theorem 0}, and let $\qq$ be a subset of $G$. 
For each $\gn$, we denote by $\qn$ the set $\qn=
W_{n}\cup Q$. Remark that $G$ is the increasing union of the sets 
$\qn$, $\gn$.
We also denote by $\varepsilon _{0}$ a positive constant such 
that assumption (\ref{Pro1}) holds true: any representation of $G$ 
admitting a $(\qq,\varepsilon _{0})$-invariant vector has a 
finite dimensional subrepresentation. 
\par\smallskip 
 In order to prove Theorem \ref{Theorem 0}, we argue by contradiction, and 
suppose that  $Q_{n}$ is a non-\ka\ set in $G$ for every $\gn$. We will 
then 
construct for every $\varepsilon >0$ a representation $\pi $ of $G$ which 
admits a $(Q,\varepsilon )$-invariant vector, but is weakly mixing (which, 
by Proposition \ref{Proposition 4.1}, is equivalent to the fact that $\pi 
$ has no finite dimensional subrepresentation), and this will contradict 
(\ref{Pro1}). 

\subsection{Construction of a sequence $(\pi _{n})_{\gn}$ of 
finite dimensional representations of $G$}\label{Section 5.2}
The first step of the proof is to show that assumption (\ref{Pro1}) 
combined with the 
hypothesis that $\qn$ is a non-\ka\ set for every $\gn$ implies the 
existence of sequences of finite dimensional representations of $G$ with 
certain properties.
\begin{lemma}\label{Lemma 6.1}
Let $\epsilon_0$ be a positive constant such that assumption (*)
holds true and suppose that $Q_n$ is a non-Kazhdan set in G for every $n \ge
1$.
For every sequence $(\varepsilon _{n})_{n\ge 1}$ of positive real numbers 
decreasing to zero with  $\varepsilon _{1}\in(0,\varepsilon _{0}]$, 
there exist a sequence $(H_{n})_{n\ge 1}$ of finite dimensional Hilbert 
spaces and a sequence $(\pi_{n})_{n\ge 1}$ of unitary representations of $G$ such that, for every $n\ge 1$, $\pi_{n}$ is a representation of $G$ on $H_n$ and
% $\pi _{n}$ is a representation of $G$ on $\hn $ with the following two properties:
\begin{enumerate}
 \item [--] $\pi _{n}$ has no non-zero $G$-invariant vector;
 \item[--] $\pi _{n}$ has a $(\qn,\en)$-invariant unit vector 
$\an\in H_{n}$: $||\an||=1$ and 
\[
\sup_{g\,\in\,\qn}
||\,\pi _{n}(g)\an-\an||<\en.
\]
\end{enumerate}
\end{lemma}
\begin{proof} Let $\gn$. Since $\qn$ is a not a \ka\ set in 
$G$, there exists a representation $\rho _{n}$ of $G$ on a Hilbert space 
$\kn$ which has no non-zero $G$-invariant vector, but is such that there 
exists a unit vector $\xn\in\kn$ with
\[
\sup_{g\,\in\,\qn} ||\,\rho _{n}(g)\xn-\xn||<2^{-n}.
\]
 Since $2^{-n}\le\varepsilon _{0}$ for $n$ large enough, assumption (\ref{Pro1}) 
implies that, for such integers $n$, $\rho _{n}$ has a finite dimensional subrepresentation. 
By Proposition \ref{Proposition 4.1}, $\rho_{n}$ is not weakly mixing. This 
means that if we decompose $K_{n}$ as $K_{n}=K_{\nw }\opl K_{\nc }$ and 
$\rho _{n}$ as $\rho _{n}=\rho  _{\nw }\opl \rho _{\nc }$, where 
$\rho _{n,w}$ and 
$\rho _{n,c}$ are respectively 
the weakly mixing and compact parts of $\pi_n$, $\rho _{\nc}$ is  non-zero. Since $\rho _{n}$ has no non-zero $G$-invariant vector, 
neither have $\rho _{\nw }$ nor $\rho _{\nc }$.
\par\smallskip 
Decomposing $\xn$ as $\xn=x_{\nw }\opl x_{\nc }$, we have
$1=||x_{\nw }||^{2}+||x_{\nc }||^{2}$. We claim that
$\underline{\lim}_{\,n\to+\infty }\,||\,x_{\nc }||>0$. Indeed, suppose 
that it is not the case. Then $\overline{\lim}_{\,n\to+\infty 
}||x_{\nw }||=1$. Since $||\rho _{n}(g)\xn-\xn||^{2}=||\rho 
_{\nw }(g)x_{\nw }-x_{\nw}||^{2}+||\rho 
_{\nc }(g)x_{\nc }-x_{\nc }||^{2}$ for every $g\in G$, we have
\begin{align*}
\sup_{g\,\in\,Q_{n}}\Bigl|\Bigl|\rho 
_{\nw }(g)\dfrac{x_{\nw }}{||x_{\nw }||}-\dfrac{x_{\nw }}{||x_{\nw }||
}\Bigr|\Bigr|<\dfrac{2^{-n}}{||x_{\nw }||} 
\end{align*}
as soon as $x_{\nw }$ is non-zero. Since 
$\overline{\lim}_{\,n\to+\infty }||x_{\nw }||=1$, this implies that for 
any $\delta >0$ there exists an integer $n$ such that $\rho _{\nw }$ has 
a $(\qn,\delta )$-invariant vector of norm $1$. 
Applying this to $\delta =\varepsilon_{0} $, there exists $n_{0}\ge 1$ such that $\rho _{n_{0},\,w}$ has a $(Q_{n_{0}}, \varepsilon _{0})$-invariant vector, hence a $(Q,\varepsilon _{0})$-invariant vector.
But
$\rho _{n_{0},\,w}$ is weakly mixing, so has no 
finite dimensional subrepresentation. This 
contradicts assumption (\ref{Pro1}). So we deduce that
$\underline{\lim}_{\,n\to+\infty }\,||\,x_{\nc }||=\gamma >0$.
The same observation as above, applied to the representation $\rho 
_{\nc }$, shows that
\begin{align*}
 \sup_{g\,\in\,Q_{n}}\Bigl|\Bigl|\rho 
_{\nc }(g)\dfrac{x_{\nc }}{||x_{\nc }||}-\dfrac{x_{\nc }}{||x_{\nc }||
}\Bigr|\Bigr|&<\dfrac{2^{-n}}{||x_{\nc }||} 
\intertext{for every $n$ such that $x_{\nc }$ is non-zero, and thus that }
\sup_{g\,\in\,Q_{n}}\Bigl|\Bigl|\rho 
_{\nc }(g)\dfrac{x_{\nc }}{||x_{\nc }||}-\dfrac{x_{\nc }}{||x_{\nc }||
}\Bigr|\Bigr|&<\dfrac{2^{-(n-1)}}{\gamma } 
\end{align*}
for infinitely many integers $n$. For these integers, $\rho _{\nc }$ is a 
compact representation for  which $\yn=x_{\nc }/||x_{\nc }||$ is a $(\qn,2^{-(n-1)}/\gamma 
)$-invariant vector of norm $1$. It has no non-zero 
$G$-invariant vector. 
Decomposing  $\rho _{\nc }$ as a direct sum of finite dimensional representations, straightforward computations show that there exists for each such integer $n$ a finite dimensional representation
$\sigma _{n}$ of $G$ with a 
$(\qq_{n},2^{-(n-2)}/\gamma )$-invariant vector but no non-zero $G$-invariant 
vector.
Lemma \ref{Lemma 6.1} follows immediately by taking a suitable subsequence of $(\sigma_{n})_{n\ge 1}$.
\end{proof}
\subsection{Construction of weakly mixing representations of $G$ with $(\qq,\varepsilon)$-invariant vectors}\label{Section 5.3}
Let $\varepsilon >0$ be an arbitrary positive number.
Our aim is to show that there exists a weakly mixing representation of $G$ with a $(\qq,\varepsilon )$-invariant vector.
We fix a 
sequence $(\en)_{\gn}$ of positive numbers decreasing to zero so fast that
the following properties hold:
\begin{enumerate}
 \item [(i)] $0<\en<\varepsilon _{0}$ for every $\gn$, and 
$\sum_{\gn}\en<\varepsilon ^{2}/2$;
\item[(ii)] the sequence $(\frac{1}{(n+1) \en^2}\sum_{j=n}^{2n}\varepsilon_j^2)_{n\ge 1}$ tends to $0$ as $n$ tends to infinity.
\end{enumerate}
\par\smallskip
We consider the representation
$\pmb{\pi} =\oti_{\gn}\,\pi _{n}$ of $G$ on the infinite tensor product  space 
$\pmb{H}=\inc{\aaaa}$, 
where the spaces $H_{n}$, the 
representations $\pi _{n}$ and the vectors $\an$ are associated to 
$\en$ for each $\gn$ by Lemma \ref{Lemma 6.1}. We refer to the appendix for undefined notation concerning infinite tensor products. We first prove the 
following fact:
\begin{fact}\label{Proposition 6.2}
 Under the assumptions above, $\pmb{\pi} $ is a strongly continuous representation 
of 
$G$ 
on $\pmb{H}$ which  has a $(\qq,\varepsilon )$-invariant 
vector.
\end{fact}
\begin{proof}[Proof of Fact \ref{Proposition 6.2}] 
 In order to prove that $\pmb{\pi} $ is well-defined and strongly continuous, it suffices 
to 
check that the assumptions of Proposition \ref{Proposition 3.2} in the appendix hold true. 
For every $g\in G$ and $\gn$, we have
$|1-\pss{\pi _{n}(g)\an}{\an}|\le||\pi _{n}(g)\an-\an||$ so that 
$$\sup_{g\,\in\,Q_{n}}|1-\pss{\pi _{n}(g)\an}{\an}|<\varepsilon _{n}.$$
By assumption (i), the series $\sum_{\gn}\en$ is convergent. Since every 
element $g\in G$ belongs to all the sets $\qn$ except finitely many, the 
series $\sum_{\gn}|1-\pss{\pi _{n}(g)\an}{\an}|$ is convergent for every 
$g\in G$. Moreover, it is uniformly convergent on $Q_{1}$, and hence on 
$W_{1}$. The function 
\[
\smash{\xymatrix{g\ar@{|->}[r]&\ds\sum_{\gn}|1-\pss{\pi _{n}(g)\an}{\an}|}}
\]
\par\bigskip
\noindent
is thus continuous on $W_{1}$, which is a neighborhood of $e$. It follows 
then from Proposition \ref{Proposition 3.2} that $\pmb{\pi } $ is strongly 
continuous on $\pmb{H}$. If $G$ is locally compact,  
Proposition \ref{Proposition 3.2.0} and the first part of the argument 
above suffice to show that $\pmb{\pi }$ is strongly continuous, even when 
 $W_{1}$ is not a neighborhood of $e$.
\par\smallskip  
Next, it is easy to check that the 
elementary vector $\aaaa=\oti_{\gn}\an$ of $\inc{\aaaa}$ satisfies
$||\aaaa||=1$ and $\sup_{g\in\qq}||\pmb{\pi } (g)\aaaa-\aaaa||<\varepsilon 
$.
 Indeed $||\aaaa||=\prod_{\gn}||\an||=1$, and for every $g\in \qq$ 
we have (using the fact that $\qq\subseteq \qn$ for every $\gn$)
\begin{align*}
 ||\pmb{\pi } (g)\aaaa-\aaaa||^{2}&=2\,(1-\textrm{Re}\pss{\pmb{\pi} 
(g)\aaaa}{\aaaa})\le 
2\,
 \Bigl|1-\prod_{\gn}\pss{\pi _{n}(g)\an}{\an}\Bigr|\\
 &\le 2\sum_{\gn}|1-\pss{\pi _{n}(g)\an}{\an}|<2\sum_{\gn}\en .
\end{align*}
Assumption (i) on the sequence $(\en)_{\gn}$ implies that 
$\sup_{g\,\in\,\qq} ||\pmb{\pi } (g)\aaaa-\aaaa||^{2}<\varepsilon ^{2}$,
and $\aaaa$ is thus a $(\qq,\varepsilon )$-invariant vector for $\pmb{\pi }$.
\end{proof}

Using the notation of Section \ref{Section 4.1}, we now decompose
$\pi _{n}$ and $H_{n}$ as
\[
\pi _{n}=\opl_{j\,\in\,J_{n}}\Bigl(\ \opl_{i\,\in\, I_{j,\,n}}\pi 
_{\jn } \Bigr)\quad\textrm{and}\quad 
H _{n}=\opl_{j\,\in\,J_{n}}\Bigl(\ \opl_{i\,\in\, I_{j,\,n}}H
_{i,\,\jn } \Bigr)
\]
respectively. Since $H_{n}$ is finite dimensional, all the sets $J_{n}$ 
and $I_{\jn }$, 
$j\in J_{n}$, are finite, and we assume that they are subsets of $\N$. For every $j\in J_{n}$, $H_{\ijn}=K_{\jn}$. 
We  also decompose $\an\in 
H_{n}$ as $\an=\opl\limits_{j\,\in J_{n}}\,\bigl(\,\,\opl\limits_{i\,\in 
I_{\jn}}a_{\ijn}\, \bigr)$, and write
$\ti{a}_{\jn}=\oplus_{i\in I_{\jn}}a_{\ijn}$ for every $j\in J_{n}$.
We have 
\begin{equation}\label{4.0}
||\ti{a}_{\jn}||=\left(\sum_{i\,\in 
I_{\jn}}\!\!||\,a_{\ijn}||^{2}\right)^{\frac{1}{2}}
\quad\textrm{ and }\quad
||a_{n}||= \left(\sum_{j\,\in J_{n}}\,\,||\ti{a}_{\jn}||^{2}\right)^{\frac{1}{2}}=1,
\end{equation}
so that $||\ti{a}_{\jn}||\le 1$ for every $j\in J_{n}$.
Also,
\smallskip
\begin{align}\label{Eq7}
 ||\pi _{n}(g)\an-\an||^{2}&=\sum_{j\,\in J_{n}}\,\,\sum_{i\,\in I_{\jn}}||
 \pi _{\jn}(g)a_{\ijn}-a_{\ijn}||^{2} \quad \textrm{ for every } g\in G,
 \end{align}
so that
\begin{equation}\label{Eq8}
 \sup_{g\,\in\,\qn}\Bigl(\sum_{i\,\in I_{\jn}}||\pi 
_{\jn}(g)\,a_{\ijn}-a_{\ijn}||^{2}\Bigr)^{1/2}<\en
\quad \textrm{ for every } j\in J_n.
\end{equation}
\par\medskip
There are now two cases to consider.
\par\medskip
\noindent $\bullet$ \textbf{Case 1.} We have $\underline{\textrm{lim}}_{n\to+\infty}\max_{j\in J_n}||\ti{a}_{\jn}||=0$.
\par\medskip
Using Corollary \ref{Corollary 5.7} and the fact that $\sum_{j\,\in 
J_{n}}||\,\ti{a}_{\jn}||^{2}=||\an||^{2}=1$, 
 we obtain that
\[
 m(\,|\pss{\pi _{n}(\,\centerdot\,)\an}{\an}|^{2})\le\sum_{j\,\in J_{n}}
 ||\,\ti{a}_{\jn}||^{4}\le \max_{j\,\in 
J_{n}}||\,\ti{a}_{\jn}||^{2}\,.\,\sum
_{j\,\in J_{n}}||\,\ti{a}_{\jn}||^{2}\le \max_{j\,\in 
J_{n}}||\,\ti{a}_{\jn}||^{2}.
\]
 It follows from our assumption that 
$\underline{\textrm{lim}}_{n\to +\infty}m(\,|\pss{\pi _{n}
(\,\centerdot\,)\an}{\an}|^{2})=0$.
So $\pmb{\pi} $ is weakly mixing by 
Proposition \ref{Proposition 4.3}. We have thus proved in this case the existence of a weakly mixing representation of $G$ with a $(\qq,\varepsilon)$-invariant vector.
\par\medskip
\noindent $\bullet$ \textbf{Case 2.} There exists $\delta>0$ such that $\max_{j\in J_n}||\ti{a}_{\jn}||>\delta$ for every $n\ge 1$.
\par\medskip
Let, for every $n\ge 1$, $j_n\in J_n$ be such that $||\ti{a}_{j_n,\,n}||>\delta$. Set $I_n=I_{j_n,\,n}\subseteq\N$, $\sigma_n=\pi_{j_n,\,n}$, $K_n=K_{j_n,\,n}$ and $b_{i,\,n}=a_{i,\,j_{n},\,n}$ for every $i\in I_n$.
Then $\sigma_n$ is a non-trivial irreducible representation of $G$ on  
the finite dimensional  space $K_{n}$, and by (\ref{4.0}) and (\ref{Eq8}) the finite family $(b_{i,\,n})_{i\,\in I_{n}}$ of vectors 
of 
$K_{n}$ satisfies
\begin{equation}\label{eq++}
1\ge\Bigl(\sum_{i\,\in I_{n}}||\,b_{i,\,n}||^{2} 
\Bigr)^{1/2}>\delta\quad\textrm{and}\quad\sup_{g\,\in\,\qn}
\Bigl(\sum_{i\,\in I_{n}}||\sigma _{n}(g)\,b_{i,\,n}-b_{i,\,n} 
||^{2}\Bigr)^{1/2}\!\!\!<\en.
\end{equation}
 If we write 
\[
\ti{K}_{n}=\opl_{i\,\in I_{n}}K_{n},\quad \ti{b}_{n}=\opl_{i\,\in 
I_{n}}b_{i,\,n},\quad \textrm{and}\quad \ti{\sigma }_{n}=\opl_{i\,\in 
I_{n}}\sigma 
_{n},
\]
this means that 
\begin{equation}\label{4.4}
 1\ge ||\ti{b}_{n}||>\delta \quad \textrm{ and } \quad \sup_{g\,\in\,\qn}||\ti{\sigma 
}_{n}(g)\ti{b}_{n}-
\ti{b}_{n}||<\en.
\end{equation}
\par\medskip 
 Now we again have to consider separately two cases.
 \par\medskip
\noindent \emph{-- Case 2.a.}
There exists an infinite 
subset $D$ of $\N$ such that 
whenever $k$ and $l$ are two distinct elements of $D$, $\sigma _{k}$ and $\sigma
 _{l}$ are not equivalent. Replacing the sequence $(\sigma
_{n})_{\gn}$ by $(\sigma 
_{n})_{n\in D}$, we can 
 suppose without loss of generality 
that for every distinct integers $m$ and $n$, with $m,n\ge 1$, $\sigma _{m}$ and $\sigma _{n}$ 
are not equivalent.
\par\smallskip 
Consider for every $\gn$ the representation
\[
\rho _{n}=\ti{\sigma }_{n}\opl\cdots\opl\ti{\sigma} _{2n}\quad 
\textrm{of}\ G\ \textrm{on}\quad
\hh _{n}=\ti{K}_{n}\opl\cdots\opl\ti{K}_{2n},
\]
and the vector $b_{n}={\bigl(\sum_{k=n}^{2n}||\ti{b}_{k}||^{2}\bigr)^{-\frac{1}{2}}}\bigl(\,\ti{b}_{n}\opl\cdots\opl 
\ti{b}_{2n} \bigr)$
of $\hh _{n}$, which satisfies $||b_{n}||=1$. For every $g\in \qn$ we 
have, since $\qn$ is contained in $Q_{j}$ for every $j\ge n$,
\[
||\rho_{n}(g)b_{n}-b_{n}||^{2}={\bigl(\sum_{k=n}^{2n}||\ti{b}_{k}||^{2}\bigr)^{-{1}}}\sum_{j=n}^{2n}\ 
\bigl|\bigl|\ti{\sigma 
}_{j}(g)\ti{b}_{j}-\ti{b}_{j}\bigr|\bigr|^{2}<\dfrac{1}{\delta^{2}(n+1)}\sum_{j=n}^{2n}
\varepsilon _{j}^{2}
\]
by (\ref{4.4}).
By assumption
(ii) on the sequence $(\en)_{\gn}$, we obtain that there exists an integer $n_{0}\ge 1$ such that
$\sup_{g\,\in\,\qn}||\rho _{n}(g)b_{n}-b_{n}||<\en$ for every $n\ge n_{0}$. Let 
now $\pmb{\rho}
=\oti_{n\ge n_{0}}\rho _{n}$ be the infinite tensor product of the representations $\rho 
_{n}$ on the space 
$\pmb{\hh} =\bigotimes_{n\ge n_{0}}^{\aaaa}\hh _{n}$. An argument similar  
to the one given in Fact \ref{Proposition 6.2} shows that
$\pmb{\rho}  $ is a strongly continuous 
representation of $G$ 
on 
 $\pmb{\hh}$ which has a $(\qq,\varepsilon )$-invariant 
vector.
It remains to 
prove that $\pmb{\rho} $ is weakly mixing, and for this we will show  that
$m(|\pss{\rho
_{n}(\,\centerdot\,)b_{n}}{b_{n}}|^{2})$ tends to zero as $n$ tends to 
infinity.
Recall that for every $\gn$, the representations $\sigma _{n},\dots,\sigma 
_{2n}$ are mutually non-equivalent, so that, by Corollary \ref{Corollary 5.7}, we 
have for every $\gn$
\begin{eqnarray*}
m(|\pss{\rho_{n}(\,\centerdot\,)b_{n}}{b_{n}}|^{2})\le
 \sum_{j=n}^{2n}\Bigl|\Bigl|{\bigl(\sum_{k=n}^{2n}||\ti{b}_{k}||^{2}\bigr)^{-\frac{1}{2}}}\ti{b}_{j} \Bigr| 
\Bigr|^{4}\le \dfrac{1}{\delta^{4}(n+1)^{2}}\sum_{j=n}^{2n} ||\ti{b}_{j}||^{4}
\le \dfrac{1}{\delta^{4}(n+1)}
\end{eqnarray*}
by (\ref{4.4}).
So $m(|\pss{\rho_{n}(\,\centerdot\,)b_{n}}{b_{n}}|^{2})$ tends to zero as $n$ tends 
to infinity. By Proposition \ref{Proposition 4.3},
$\pmb{\rho }$ is weakly mixing. We have proved again in this case the existence of a weakly mixing representation of $G$ with a $(\qq,\varepsilon)$-invariant vector.
\par\medskip

The other case we have to consider is when there exists an integer $n_{1}\ge 1$ such that for every $n\ge n_{1}$, $\sigma_{n}$ is equivalent to one of the representations $\sigma _{1},\dots,\sigma _{n_{1}}$. Indeed, if there is no such integer, we can construct a strictly increasing sequence $(n_{k})_{k\ge 1}$ of integers such that, for every $k\ge 1$, $\sigma_{n_{k}}$ is not equivalent to one of the representations $\sigma _{1},\dots,\sigma _{n_{k-1}}$. The set $D=\{n_{k}\, ;\, k\ge 1\}$ then has the property that whenever $m$ and $n$ are two distinct elements of $D$, $\sigma_{m}$ and $\sigma_{n}$ are not equivalent, and we are back to the setting of Case $2.a$. Without loss of generality, we can suppose that $\sigma_n$ is equal to $\sigma_{1}$ for every $n\ge 1$.
\par\medskip
\noindent \emph{-- Case 2.b.} For every $n\ge 1$, $\sigma_n$ is equal to $\sigma_{1}$.
By (\ref{eq++}), 
we have
\[
1\ge\Bigl(\sum_{i\,\in I_{n}}||\,b_{i,\,n}||^{2} 
\Bigr)^{1/2}>\delta\quad\textrm{and}\quad\sup_{g\,\in\,\qn}
\Bigl(\sum_{i\,\in I_{n}}||\sigma _{1}(g)b_{i,\,n}-b_{i,\,n} 
||^{2}\Bigr)^{1/2}\!\!<\en,
\]
where all the vectors $b_{i,n}$, $i\in I_{n}$, belong to $H_{1}$.
For each $\gn$, set $c_{n}=\opl\limits
_{i\,\in I_{n}} b_{i,\,n}$, seen as a vector of the infinite direct 
sum
$H=\opl\limits_{j\ge 1}H_{1}$ by defining its $j^{th}$ coordinate to be zero when $j$ does not belong to $I_{n}$. Let also $\sigma $ be the infinite direct sum
$\sigma =\opl\limits_{j\ge 1}\sigma _{1}$ of $\sigma _{1}$ on $H$. Then we have, 
for every $\gn$, 
\[
1\ge||c_{n}||>\delta\quad\textrm{and}\quad\sup_{g\,\in\,\qn} ||\sigma
(g)c_{n}-c_{n}||<\en.
\]
\par\smallskip 
Let now $S$ be a finite subset of $G$. There exists an integer $n_{S}\ge 
1$ such that $S\subseteq \qn$ for every $n\ge n_{S}$, and hence
\[
\sup_{g\,\in S}\,||\sigma (g)c_{n}-c_{n}||<\en\quad\textrm{for every $n\ge 
n_{S}$.}
\]
It follows that $\sigma $ has almost-invariant vectors for finite sets: for every $\delta >0$ and every finite 
subset $S$ of $G$, $\sigma $ has an $(S,\delta )$-invariant vector. This 
implies that $\sigma_{1}$ itself has almost-invariant vectors for finite 
sets (see \cite[Lem.~1.5.4]{Pet} or \cite{Ke}). 
Since $\sigma
_{1}$ is a finite dimensional representation, it follows that $\sigma_{1}$ 
has almost-invariant vectors.
If $(v_n)_{n\ge 1}$ is a sequence of unit vectors of $H_1$ such that 
\[
\sup_{g\,\in G}\,||\sigma (g)v_{n}-v_{n}||<2^{-n}\quad\textrm{for every $n\ge 1$},
\]
then any accumulation point of  $(v_n)_{n\ge 1}$ is a
non-zero $G$-invariant vector for $\sigma_{1}$. This contradicts our initial assumption on 
$\sigma_{1}$, and shows that the hypothesis of Case 2.b cannot be fulfilled.
\par\medskip
Summing up our different cases, we have thus proved that there exists for every $\varepsilon 
>0$ a representation of $G$ with a $(Q,\varepsilon )$-invariant vector but 
no finite dimensional subrepresentation. This contradicts  assumption (\ref{Pro1}) of Theorem 
\ref{Theorem 0}, and concludes the proof.

 \section{Some consequences of Theorem \ref{Theorem 0}}\label{Section 6}
 We begin this section by proving the two characterizations of \ka\ sets obtained as consequences of Theorem \ref{Theorem 0}.
 
 \subsection{Proofs 
of Corollary \ref{Corollary 1} and Theorem \ref{Theorem 2}} \label{Section 
7.2}
Let us first prove Corollary 
\ref{Corollary 1}.
\begin{proof}[Proof of Corollary \ref{Corollary 1}]
Let $\qq_{0}$ 
be a subset of $G$ which has non-empty interior and which generates $G$.
Denote for each $\gn$ by $\qq_{0}^{\,\pm n}$ the set $\{g_{1}^{\,\pm 1}
\dots g_{n}^{\,\pm 1}\,;\, g_{1},\dots,g_{n}\in\qq_{0}\}$. Then 
$G=\bigcup_{\gn}\qq_{0}^{\,\pm n}$. Let $g_{0}$ be an element of the 
interior of 
 $\qq_{0}$. Then $g_{0}^{-1}\qq_{0}$ is a neighborhood of $e$. There 
exists $n_{0}\ge 1$ such that $g_{0}^{-1}$ belongs to 
$\qq_{0}^{\,\pm n_{0}}$, 
and thus $\qq_{0}^{\,\pm(n_{0}+1)}$ is a neighborhood of $e$. If we set
$W_{n}=\qq_{0}^{\,\pm(n_{0}+n)}$ for $\gn$, the sequence of sets 
$(W_{n})_{\gn}$ is increasing, $W_{1}$ is a neighborhood of $e$, and $(W_{n})_{n\ge 1}$
satisfies the assumptions of Theorem \ref{Theorem 0}. So if $\qq$ is a 
subset of $G$ for which assumption (\ref{Pro1}) of Theorem \ref{Theorem 0} 
holds true, there exists 
$n\ge 1$ such that $\qq_{0}^{\,\pm(n+n_{0})}\cup Q$ is a \ka\ set in $G$. 
Let $\varepsilon >0$ be a \ka\ constant for this set. Then $\varepsilon 
/(n+n_{0})$ is a \ka\ constant for $\qq_{0}\cup\qq$,
and $\qq_{0}\cup \qq$ is a \ka\ set in 
$G$. If $G$ is locally compact, the same proof 
holds true without the assumption that $\qq_{0}$ has non-empty interior.
\end{proof}

\begin{proof}[Proof of Theorem \ref{Theorem 2}]
Let us first show that (a) implies (b). Suppose that $Q$ is a \ka\ set, and let $0<\varepsilon <\sqrt{2}$ be a \ka\ constant for $Q$. Let 
$\delta = \sqrt{1-{\varepsilon^2}/{2}}$ 
and consider a  representation $\pi$ of $G$ on a Hilbert space $H$ for 
which there is a vector $x\in H$ with $||x||=1$ such that 
$\inf_{g\in\qq}\left|\pss{\pi(g)x}{x}\right|  > \delta.$
Then the representation $\pi \oti \ba{\pi }$ of 
$G$ on $H\oti \ba{H}$ verifies
\[
2\textrm{Re}\pss{\pi \oti\ba{\pi }(g)x \oti\ba{x }}{x 
\oti\ba{x }}=2|\pss{\pi (g)x }{x }|^{2} > 2 - \varepsilon^2
\]
for every $g\in\qq$. Hence
$
\left\|\pi \oti\ba{\pi }(g)x \oti\ba{x } - x \oti\ba{x }\right\| < \varepsilon
$
for every $g\in\qq$ and $\pi \oti \ba{\pi }$ has a non-zero \mbox{$G$-invariant} vector. It follows from Proposition \ref{Proposition 4.1}  that $\pi$ has a finite dimensional subrepresentation. Thus $(b)$ is true.
That (b) implies (c) is straightforward, and
that (c) implies (a) is a consequence of Corollary \ref{Corollary 1}.
\end{proof}

\subsection{Property (T) in $\sigma $-compact 
locally compact groups}\label{Section7.1}
As a  consequence of Theorem \ref{Theorem 0}, we retrieve a 
characterization of Property (T) due to Bekka and Valette \cite{BV}, 
\cite[Th.~2.12.9]{BdHV}, valid for $\sigma $-compact locally compact  
groups, which
states the following:
\begin{theorem}[\cite{BV}]\label{Theorem 7.1}
Let $G$ be a $\sigma $-compact locally compact group. 
Then $G$ has 
Property (T) if and only if every unitary representation of $G$ with 
almost-invariant vectors has a non-trivial finite dimensional 
subrepresentation.
\end{theorem}
The proof of \cite{BV} relies on the equivalence between Property (T) and 
Property (FH) for such groups \cite[Th.~2.12.4]{BdHV}. As a 
direct consequence of Theorem \ref{Theorem 0}, we will derive a new proof of 
Theorem \ref{Theorem 7.1} which does not involve property (FH).
\par\smallskip
If $\qq$ is a subset of a topological group $G$, and if $\pi $ is a unitary representation of $G$ on a Hilbert space $H$, we say that $\pi $ has \emph{$\qq$-almost-invariant vectors} if it has $(\qq,\varepsilon )$-invariant vectors for every $\varepsilon >0$. The same argument as in  \cite[Prop.~1.2.1]{BdHV} shows that $\qq$ is a \ka\ set in $G$ if and only if every representation of $G$ with $\qq$-almost-invariant vectors has a non-zero $G$-invariant vector. As a direct corollary of Theorem \ref{Theorem 2}, we obtain the following characterization of \ka\ sets which generate the group:
\begin{corollary}\label{CorollaryA}
Let $\qq$ be a subset of a locally compact group $G$ which generates $G$. Then $\qq$ is a \ka\ set in $G$ if and only if every representation $\pi $ of $G$ with $\qq$-almost-invariant vectors has a non-trivial finite dimensional subrepresentation.
\end{corollary}
\begin{proof}[Proof of Corollary \ref{CorollaryA}]
 The only thing to prove is that if every representation $\pi $ of $G$ with $\qq$-almost-invariant vectors has a non-trivial finite dimensional representation, $\qq$ is a \ka\ set. For this it suffices to show the existence of an $\varepsilon >0$ such that assumption (\ref{Pro1}) of Theorem \ref{Theorem 0} holds true. 
 The argument is exactly the same as the one given in \cite[Prop.~1.2.1]{BdHV}: suppose that there is no such $\varepsilon $, and let, for every $\varepsilon >0$, $\pi _{\varepsilon }$ be a representation of $G$ with a $(\qq,\varepsilon )$-invariant vector but no finite dimensional subrepresentation. Then $\pi =\bigoplus_{\varepsilon >0}\pi _{\varepsilon }$ has $\qq$-almost-invariant vectors but no finite dimensional subrepresentation (this follows immediately from \cite[Prop.~A.1.8]{BdHV}), contradicting our initial assumption. 
 \end{proof}
 
\begin{proof}[Proof of Theorem \ref{Theorem 7.1}]
It is clear that Property (T) implies that every representation of $G$ with almost-invariant vectors 
has a non-trivial finite dimensional subrepresentation.
 Conversely, suppose that every representation of $G$ with almost-invariant vectors 
has a non-trivial finite dimensional subrepresentation. Using the same argument as in the proof of Corollary \ref{CorollaryA}, we see that there exists 
a compact subset $\qq$ of $G$ such that assumption (\ref{Pro1}) of 
Theorem \ref{Theorem 0} holds true. 
Choosing for $(W_{n})_{\gn}$ an increasing sequence of compact subsets of $G$ 
such that $\bigcup_{\gn}W_{n}=G$, Theorem \ref{Theorem 
0} implies that there exists an $\gn$ such that 
$W_{n}\cup\qq$ is a \ka\ set in $G$. Since $W_{n}\cup\qq$ is compact, 
$G$ has Property (T).
\end{proof}

\subsection{Equidistribution assumptions: proofs of Theorems 
\ref{Theorem 1} and \ref{Theorem 
3}} Let $G$ be a second countable locally compact 
group, and let $\pi $ be a unitary 
representation of $G$ on a separable Hilbert space $H$. Such a 
representation 
can be decomposed as a direct integral of irreducible unitary 
representations over a Borel space (see for instance 
\cite[Sec.~F.5]{BdHV} or \cite{Fo}). More precisely, there exists a finite positive 
measure $\mu $ on a standard Borel space $Z$, a measurable field
$z\mapsto H_{z}$ of Hilbert spaces over 
$Z$, and a measurable field of irreducible representations 
$z\mapsto \pi_{z}$, where each $\pi _{z}$ is a representation of $G$ on $H_{z}$,
such that $\pi $ is unitarily equivalent to the direct integral
$\pi _{\mu }=\ds\int_{Z}^{\opl}\pi_{z}\,d\mu (z)$ on $\hh=\ds\int_{Z}^{\opl}H_{z}\,d\mu 
(z)$. The Hilbert space
$\hh$ is the set of equivalence classes of square integrable vector 
fields $z\mapsto x_z$, with $x_z\in H_{z}$, 
with respect to the measure $\mu $;
$\pi _{\mu }$ is the 
representation of $G$ on $\hh$ defined by 
$
\pi _{\mu }(g)x=[z\mapsto\pi_{z}(g)x_z]
$
for every $g\in G$ and $x\in\hh$. 
\begin{proof}[Proof of Theorem \ref{Theorem 3}] Our aim is to show that, 
under the hypothesis of Theorem \ref{Theorem 3}, assumption (\ref{Pro1}) 
of Theorem \ref{Theorem 0} is satisfied. Let $\pi $ be a representation of $G$ on 
a Hilbert space $H$. Since $G$ is second countable, we can suppose that 
$H$ is separable. Suppose that $\pi $ admits a $(\qq,1/2
)$-invariant vector $x \in H$ and, using the notation and the result recalled above, write 
\[
\pi =\int_{Z}^{\opl}\pi_{z}\,d\mu (z),\quad x 
=[\xymatrix{\!z\ar@{|->}[r]&x_{z}\!}],\quad \textrm{and}\quad
H=\int_{Z}^{\opl}H_{z}\,d\mu (z).
\]
We have for 
every $k\ge 1$
\begin{align}
\textrm{Re}\,\pss{\pi (g_{k})
 x }{x }&=\textrm{Re}\int_{Z}\pss{\pi_{z}(g _{k})
 x_{z} }{x_{z}}\,d\mu (z)=1-\frac{1}{2}\,||\pi (g_{k})x -x ||^{2}>\frac{7}{8}\notag
 \intertext{so that}
 &\textrm{Re}\int_{Z}\dfrac{1}{N}\sum_{k=1}^{N}\pss{\pi_{z}(g_{k})x_{z}}{x_{z}}\,d\mu 
(z)> \frac{7}{8}\quad \textrm{for every}\, N\ge 1. \label{Eq10}
\end{align}
Now, assumption (\ref{Pro3}) of Theorem \ref{Theorem 3} states that there exists a countable set $\mathcal{C}_0$ of equiva\-lence classes of irreducible representations such that
 \begin{equation}\label{5.1}
 {\dfrac{1}{N}\ds\sum_{k=1}^{N}\pss{\pi (g_{k})x}{x }\longrightarrow 0\quad\textrm{ as } N\longrightarrow +\infty}
 \end{equation}
 for every irreducible representation $\pi$ whose equivalence class $[\pi]$ does not belong to $\mathcal{C}_0$ and every vector $x$ in the underlying Hilbert space.  It follows from (\ref{5.1}) that the set $Z_0=\{z\in Z\textrm{ ; } [\pi_z]\in\mathcal{C}_0\}$ satisfies $\mu(Z_0)>0$, and there exists $[\pi_0]\in\mathcal{C}_0$ such that $\mu(\{z\in Z\textrm{ ; }\pi_z \textrm{ and }\pi_0 \textrm{ are equivalent}\})>0$. Hence $\pi_0$ is a subrepresentation of $\pi$.
Since all irreducible representations of 
$G$ are supposed to be finite dimensional, $\pi $ has a finite dimensional 
subrepresentation. So assumption (\ref{Pro1}) of Theorem \ref{Theorem 0} 
is satisfied. As $\qq$ generates $G$, it now follows from 
Theorem \ref{Theorem 2} that $\qq$ is a \ka\ set in $G$.
\end{proof}
\begin{proof}[Proof of Theorem \ref{Theorem 1}]
The proof of Theorem \ref{Theorem 1} is exactly the same as that of 
Theorem \ref{Theorem 3}, using the fact that if $G$ is a locally compact 
abelian group (not necessarily second countable), any unitary 
representation of $G$ is equivalent to a direct integral of irreducible 
representations (see for instance \cite[Th.~7.36]{Fo}).
\end{proof}

\section{Examples and applications}\label{Section 7}
We present in this section some examples of \ka\ sets in different kinds of groups, some statements being obtained as consequences of Theorems \ref{Theorem 0} or \ref{Theorem 2}. We do not try to be exhaustive, and our aim here is rather to highlight some interesting phenomena which appear when looking for \ka\ sets, as well as the connections of these phenomena with some remarkable properties of the group. We begin with the simplest case, that of locally compact abelian (LCA) groups.
\subsection{\ka\ sets in locally compact abelian groups}\label{Subsection 8.1}
Let $G$ be a second countable LCA group, the dual group of which we denote by $\Gamma$.  If $\sigma $ is a finite Borel measure on $\Gamma $, recall that its Fourier-Stieljes transform is defined by 
\[
\widehat{\sigma }(g)=\int_{\Gamma }\gamma (g)\,d\sigma (\gamma )\quad\textrm{for every}\ g\in G.
\] 
It is an easy consequence of the spectral theorem for unitary representations that if $Q$ is a subset of  a second countable LCA group $G$, $\qq$ is a \ka\ set in $G$ if and only if there exists $\varepsilon >0$ such that any probability measure $\sigma $ on $\Gamma $ with $\sup_{g\in\qq}|\widehat{\sigma }(g)-1|<\varepsilon$ satisfies $\sigma (\{{1}\})>0$, where ${1}$ denotes the trivial character on $G$.
Using Theorem \ref{Theorem 2} combined with the spectral theorem for unitary representations again, we obtain the following stronger characterization of \ka\ sets which generate the group
in any second countable LCA group. 

\begin{theorem}\label{nimporte}
 Let $G$ be a second countable LCA group, and let $\qq$ a subset of $G$ which generates $G$. The following assertions are equivalent:
 \begin{enumerate}
 \item [(1)] $\qq$ is a \ka\ set in $G$;
\item[(2)] there exists $\delta\in (0,1)$ such that any probability measure $\sigma $ on $\Gamma $ with $\inf_{g\in\qq}|\widehat{\sigma }(g)|>\delta$
has a discrete part;
\item[(3)] there exists $\varepsilon >0$ such that any probability measure $\sigma $ on $\Gamma $ with $\sup_{g\in\qq}|\widehat{\sigma }(g)-1|<\varepsilon$
has a discrete part.
\end{enumerate}
\end{theorem}

Theorem \ref{nimporte} becomes particularly meaningful in the case of the group $\Z$, as it yields a characterization of \ka\ subsets of $\Z$ involving some classic sets in harmonic analysis, introduced by Kaufman in \cite{Ka1}. They are called \emph{w-sets} by Kaufman \cite{Ka2}, and \emph{Kaufman sets} (\textbf{Ka} sets) by other authors, such as Hartman \cite{Ha1}, \cite{Ha2}.

\begin{definition}
Let $\qq$ be a subset of $\Z$, and let $\delta\in (0,1)$.

$\bullet$ We say that $\qq$ belongs to the class \textbf{Ka} if there exists a finite complex-valued  continuous Borel measure $\mu$ on $\T$ such that $\inf_{n\in\qq}|\hat{\mu }(n)|>0$, and to the class  \textbf{$\delta$-Ka} if there exists a finite complex-valued  continuous Borel measure $\mu$ on $\T$ with $\mu(\T)=1$ such that $\inf_{n\in\qq}|\hat{\mu }(n)|>\delta$.

$\bullet$ We say that $\qq$ belongs to the class \textbf{Ka}$^{+}$ if there exists a   continuous probability measure $\sigma$ on $\T$ such that $\inf_{n\in\qq}|\hat{\sigma}(n)|>0$, and to the class  \textbf{$\delta$-Ka$^{+}$} if there exists a  continuous probability measure $\sigma$ on $\T$  such that $\inf_{n\in\qq}|\hat{\sigma }(n)|>\delta$.
\end{definition}

Our characterization of \ka\ subsets of $\Z$ is given by Theorem \ref{nimportebis} below:

\begin{theorem}\label{nimportebis}
 Let $\qq$ a subset of $\Z$ which generates $\Z$. Then $\qq$ is a \ka\ set in $\Z$ if and only if there exists a $\delta\in(0,1)$ such that $\qq$ does not belong to \textbf{$\delta$-Ka$^{+}$}.
\end{theorem}

It is interesting to remark \cite{Ha2} that a set $\qq$ belongs to \textbf{Ka} if and only if it belongs to \textbf{$\delta$-Ka} for every $\delta\in (0,1)$. There is no similar statement for the class \textbf{Ka}$^{+}$: any sufficiently lacunary subset of $\Z$, such as $\qq=\{3^k+k \; ;\; k\ge 1\}$, is easily seen to belong to \textbf{Ka}$^{+}$ (it suffices to consider an associated Riesz product -- see for instance \cite{HMP} for details);
but the same reasoning as in Example \ref{Example B} below shows that this set $\qq$ is a \ka\ subset of $\Z$. Thus there exists by Theorem \ref{nimportebis} a $\delta\in (0,1)$ such that $\qq$
does not belong to \textbf{$\delta$-Ka$^{+}$}.
\par\smallskip
%Theorem \ref{Theorem 4} is another powerful  tool for exhibiting \ka\ sets in $\Z^{d}$ or $\R^{d}$. 
We present now some typical examples of \ka\ sets in $\Z$ or $\R$ obtained using the above characterizations. The first one provides a negative answer to Question \ref{Question 1} (a).
\begin{example}\label{Example B}
 The set $\qq=\{2^{k}+k\,;\,k\ge 0\}$ is a \ka\ set in $\Z$ and there are irrational numbers $\theta$ such that $(e^{2i\pi n\theta })_{n\in\qq}$ is not dense in $\T$. 
 In particular, no rearrangement $(m_k)_{k\,\ge 1}$ of the elements of $Q$ exists such that $(e^{2i\pi m_{k}\theta })_{k\,\ge 1}$ is equidistributed in 
$\T$ for every irrational number $\theta$.
 \end{example}
\begin{proof}{}
The sequence $(n_{k})_{k\ge 0}$ defined by $n_{k}=2^{k}+k$ for every $k\ge 0$ satisfies the relation $2n_{k}=n_{k+1}+k-1$ for every $k\ge 0$. 
Let $\sigma $ be a probability measure on $\T$ such that $\sup_{k\ge 0}|\wh{\sigma }(n_{k})-1|<1/18$. Since, by the Cauchy-Schwarz inequality,
  \[
  |\wh{\sigma }(k)-1|\le\int_{\T}|\lambda^{k}-1|d\sigma(\lambda)\le\sqrt{2}\,|\wh{\sigma }(k)-1|^{1/2}\quad\textrm{ for every } k\in\Z,
  \]
  we have
 \begin{align*}
|\wh{\sigma }(k-1)-1|&\le 2 \int_{\T}|\lambda^{n_{k}}-1|d\sigma(\lambda)
+\int_{\T}|\lambda^{n_{k+1}}-1|d\sigma(\lambda)\\
&\le 2\sqrt{2}\,
|\wh{\sigma }(n_{k})-1|^{1/2}+\sqrt{2}\,|\wh{\sigma }(n_{k+1})-1|^{1/2}
\end{align*}
for all $k\ge 1$,
so
that $\sup_{k\ge 0}|\wh{\sigma }(k)-1|<1$. Since
\[
\dfrac{1}{N}\sum_{k=1}^{N}\wh{\sigma }({k})=\int_{\T}\Bigl( \dfrac{1}{N}\sum_{k=1}^{N}\lambda ^{{k}}\Bigr)d\sigma (\lambda )\!\xymatrix@C=17pt{\ar[r]&}\!\sigma (\{1\}) \;\; \textrm{ as } N\!\xymatrix@C=17pt{\ar[r]&}\!+\infty,
\]
we have $\sigma(\{1\})>0$.
So $Q=\{n_{k}\,;\, k\ge 0\}$ is a \ka\ set in $\Z$. But $(n_{k})_{k\ge 0}$ being lacunary, it follows from a result proved independently by Pollington \cite{Pol} and De Mathan \cite{DM} that there exists a subset $A$ of $[0,1]$ of Hausdorff measure $1$ such that for every $\theta $ in $ A$, the set
 $Q\theta = \{n_{k}\theta \textrm{ ; } k\ge 0\}$ is not dense modulo $1$. One of these numbers $\theta $ is irrational, and the conclusion follows.
\end{proof}

\begin{example}\label{Example Bbis}
The set $\qq'=\{2^{k}\,;\,k\ge 0\}$ is not a \ka\ set in $\Z$.
\end{example}
\begin{proof}{}
The fact that $\qq'$ is not a \ka\ set in $\Z$ relies on the observation that $2^{k}$ divides $2^{k+1}$ for every $k\ge 0$. Using the same construction as the one of \cite[Prop.~3.9]{EG}, we consider for any fixed $\varepsilon>0$ a decreasing sequence  $(a_{j})_{j\ge 1}$ of positive real numbers with $a_{1}<\varepsilon/(2\pi) $ such that the series $\sum_{j\ge 1}a_{j}$ is divergent. Then the infinite convolution of two-points Dirac measures
\[
\sigma = \underset {j\geq 1}\Asterisk\bigl( (1-a_{j})\delta _{\{1\}}+a_{j}\delta _{\{e^{i\pi 2^{-j+1}}\}}\bigr)
\]
is a well-defined probability measure on $\T$, which is continuous by the assumption that the series $\sum_{j\ge 1}a_{j}$ diverges. For every $k\ge 0$,
\[
\wh{\sigma }(2^{k})=\prod_{j\ge 1}\bigl( 1-a_{j}+a_{j}e^{i\pi 2^{k-j+1}}\bigr)
=\prod_{j\ge k+1}\bigl( 1-a_{j}(1-e^{i\pi 2^{k-j+1}})\bigr).
 \]
 As $| 1-a_{j}(1-e^{i\pi 2^{k-j+1}})|\le 1$, it follows that 
\[
|\wh{\sigma }(2^{k})-1|\le\sum_{j\ge k+1}a_{j}|1-e^{i\pi 2^{k-j+1}}|\le \pi \,a_{k+1}\,2^{k+1}\sum_{j\ge k+1}2^{-j}=2\pi a_{k+1}<\varepsilon
\]
for every  $k\ge 0$. 
This proves that $\qq'$ is not a \ka\ set in $\Z$.
\end{proof}

\begin{example}\label{Example C} 
 If $p$ is a non-constant polynomial with integer coefficients such that $p(\Z)$ is included in $a\Z$ for no integer $a$ with $|a|\ge 2$, then $\qq=\{p(k)\,;\,k\ge 0\}$ is a \ka\ set in $\Z$.
\end{example}

 \begin{proof}
 Our assumption that $p(\Z)$ is included in $a\Z$ for no integer $a$ with $|a|\ge 2$ implies that $Q$ generates $\Z$. Since
the sequence $(\lambda ^{p(k)})_{k\ge 0}$ is uniformly distributed in $\T$ for every  $\lambda =e^{2i\pi \theta }$ with $\theta $ irrational (see for instance \cite[Th. 3.2]{KuiNied}), Theorem \ref{Theorem 3} implies that $Q$ is a Kazhdan set in $\Z$.
\end{proof}

\begin{example}\label{Example 8.5}
 Let $p$ be a non-constant real polynomial, and let $\qq=\{p(k)\,;\,k\ge 0\}$. Then $(-\delta ,\delta )\cup \qq$ is a \ka\ subset of $\R$ for any $\delta >0$.
\end{example}
\begin{proof}
 Write $p$ as $p(x)=\sum_{j=0}^{d}a_{j}x^{j}$, $d\ge 1$, and let $r\in\{1,\dots,d\}$ be such that $a_{r}\neq 0$. It is well-known (see for instance \cite[Th. 3.2]{KuiNied}) that the sequence $(e^{2i\pi tp(k)})_{k\in\Z}$ is uniformly distributed in $\T$ as soon as  $ta_{r}$ is irrational. This condition excludes only countably many values of $t$. Set now $W_{n}=(-n,n)$ for every integer $\gn$. Thanks to Theorem \ref{Theorem 0}, we obtain that there exists $\gn$ such that $(-n,n)\cup\qq$ is a \ka\ set in $\R$. Let $\varepsilon >0$ be a \ka\ constant for this set. Fix $\delta >0$. In order to prove that $(-\delta ,\delta )\cup\qq$ is a \ka\ set in $\R$, we consider a positive number $\gamma $, which will be fixed  later on, and let $\sigma $ be a probability measure on $\R$ such that $\sup_{t\in(-\delta ,\delta )\cup\qq}\,\bigl|\widehat{\sigma }(t)-1\bigr|<\gamma $.
For any $a\in\N$ and any $t\in(-\delta ,\delta )$,
\[
 2(1-\textrm{Re}\,\widehat{\sigma }(at))=\int_{\R}\bigl|e^{iatx}-1\bigr|^{2}d\sigma (x)\le a^{2}\int_{\R}\bigl|e^{itx}-1\bigr|^{2}d\sigma (x)\le 2a^{2} \textrm{Re}\,(1-\widehat{\sigma }(t))
\]
so that $\sup_{t\in(\delta ,\delta )}\,(1-\textrm{Re}\,\widehat{\sigma }(at))<a^{2}\gamma $. If we choose $a>n/\delta $ and $\gamma <\min(\varepsilon ,\varepsilon ^{2}/(2a^{2}))$, we obtain that 
$
\sup_{t\in(-n,n)\cup\qq}\,\bigl|1-\widehat{\sigma }(t)\bigr|<\varepsilon,
$
and since $\varepsilon $ is a \ka\ constant for $(-n,n)\cup\qq$, $\sigma (\{0\})>0$. Hence $\gamma $ is a \ka\ constant for $(-\delta ,\delta )\cup\qq$.
\end{proof}

\begin{remark}\label{Remark 8.6}
 It is necessary to add a small interval to the set $\qq$ in order to turn it into a \ka\ subset of $\R$, even when $\qq$ generates a dense subgroup of $\R$. Indeed, consider the polynomial $p(x)=x+\sqrt{2}$. The set $\qq=\{k+\sqrt{2}\,;\,k\ge 0\}$ is not a \ka\ set in $\R$: for any $\varepsilon >0$, let $b\in\N$ be such that $|e^{2i\pi b\sqrt{2}}-1|<\varepsilon $. The measure 
$\sigma$ defined as the Dirac mass at the point ${2\pi b}$ satisfies $\sup_{k\ge 0}\,|\widehat{\sigma }(k+\sqrt{2})-1|<\varepsilon $, so that $\qq$ is not a \ka\ set in $\R$.\end{remark}

 We finish this section by exhibiting a link between \ka\ subsets of $\Z^{d}$ and 
 \ka\ subsets of $\R^{d}$, $d\ge 1$. Let $Q$ be a subset of $\Z^{d}$. Seen as a subset of $\R^{d}$, $Q$ is never a \ka\ set. But as a consequence of Theorem \ref{Theorem 0}, we see that $Q$ becomes a \ka\ set in $\R^{d}$ if we add a small perturbation to it.
 
\begin{proposition}\label{Proposition D}
 Fix an integer $d\ge 1$, and let $(W_{n})_{\gn}$ be an increasing sequence of subsets of $\R^{d}$ such that $\bigcup_{\gn}W_{n}=\R^{d}$. Let $\qq$ be a \ka\ subset of $\Z^{d}$. There exists an $\gn$ such that  $W_{n}\cup \qq$ is a \ka\ set in $\R^{d}$. Also, $B(0,\delta )\cup Q$ is a \ka\ subset of $\R^{d}$ for any $\delta >0$, where $B(0,\delta )$ denotes the open unit ball of radius $\delta $ for the Euclidean norm on $\R^{d}$.
\end{proposition}

\begin{proof}
 Let $\varepsilon >0$ be a \ka\ constant for $\qq$, seen as a subset of $\Z^{d}$. Let $\pi $ be a representation of $\R^{d}$ on a separable Hilbert space $H$ which admits a $(\qq,\varepsilon ^{2}/2)$-invariant vector $x\in H$. 
Without loss of generality we can suppose that $\pi $ is a direct integral on a Borel space $Z$, with respect to a finite measure $\mu $ on $Z$, of a family $(\pi _{z})_{z\in\Z}$ of irreducible representations of $\R^{d}$. So $\pi $ is a representation of $\R^{d}$ on $L^{2}(Z,\mu )$. We write elements $f$ of $L^{2}(Z,\mu )$ as $f=(f_{z})_{z\in Z}$.
We suppose that $||x||=1$; our hypothesis implies that
\[
\sup_{\ttt\in\qq}\,\bigl|1-\pss{\pi (\ttt)x}{x}\bigr|<\dfrac{\varepsilon ^{2}}{2}\cdot 
\]
Each representation $\pi _{z}$ acts  on vectors $\ttt=(t_{1},\dots,t_{d})$ of $\R^{d}$ as 
$\pi _{z}(\ttt)=\exp(2i\pi \pss{\ttt}{\tttheta _{z}})$ for some vector $\tttheta _{z}=(\theta _{1,z},\dots,\theta _{d,z})$ of $\R^{d}$. Hence
\[
\sup_{\ttt\in\qq}\,\Bigl|1-\int_{Z}e^{2i\pi \pss{\ttt}{\tttheta _{z}}}|x_z|^{2}d\mu (z)\Bigr|<\dfrac{\varepsilon ^{2}}{2}\cdot 
\]
Consider now the representation $\rho $ of $\Z^{d}$ on $L^{2}(Z,\mu )$ defined by 
$\smash{\xymatrix{\rho (\nnn)\,f:z\ar@{|->}[r]&e^{2i\pi \pss{\nnn}{\tttheta _{z}}}f_{z}}\!}$ for every 
$\nnn=(n_{1},\dots,n_{d})\in\Z^{d}$ and every $f\in L^{2}(Z,\mu )$. We have
\[
\sup_{\nnn\in\qq}||\rho (\nnn)x-x||^{2}\le 2\sup_{\nnn\in\qq}\bigl|1-\pss{\rho (\nnn)x}{x}\bigr|<\varepsilon ^{2},
\]
and since $\varepsilon $ is a \ka\ constant for $\qq$ as a subset of $\Z^{d}$, $\rho $ has a non-zero $\Z^{d}$-invariant vector. There exists hence $f\in L^{2}(Z,\mu )$ with $||f||=1$ such that 
$\rho (\nnn)f=f$ for every $\nnn\in\Z^{d}$. Fix a representative of $f\in L^{2}(Z,\mu )$, and set 
$Z_{0}=\{z\in Z\,;\, f_{z}\neq 0\}$. Then $\mu (Z_{0})>0$. For every $z\in Z_{0}$ we have 
$e^{2i\pi \pss{\nnn}{\tttheta _{z}}}=1$ for every $\nnn\in \Z^{d}$, which implies that 
$\tttheta _{z}\in\Z^{d}$. For each $\nnn=(n_{1},\dots,n_{d})\in \Z^{d}$, let 
$Z_{\nnn}=\{z\in Z_{0}\,;\,\theta _{i,z}=n_{i}\ \textrm{for each}\ i\in\{1,\dots,n\}\}$ and
$H_{\nnn}=\{f\in L^{2}(Z,\mu )\,;\,f=0\ \mu\textrm{-a.\,e. on}\ Z\setminus Z_{\nnn}\}$. We have
$\bigcup_{\nnn\in\Z^{d}}Z_{\nnn}=Z_{0}$, so there exists $\nnn_{0}\in\Z^{d}$ such that $\mu (Z_{\nnn_{0}})>0$. Each subspace $H_{\nnn}$ is easily seen to be invariant for $\pi $, and the representation $\pi _{\nnn}$ induced by $\pi $ on $H_{\nnn}$ is given by $\smash{\xymatrix{\pi _{\nnn}(\ttt)\,f\, :\,z\ar@{|->}[r]&e^{2i\pi \pss{\ttt}{\nnn}}f_{z}}}$ for every $\ttt\in\R^{d}$ and every $f\in H_{\nnn}$. So $\pi $ admits a subrepresentation of dimension $1$ as soon as $H_{\nnn}$ is non-zero, i.\,e.\ as soon as $\mu (Z_{\nnn})>0$. Since $\mu (Z_{\nnn_{0}})>0$, $\pi$  admit a subrepresentation  of dimension $1$. An application of Theorem \ref{Theorem 2} now shows that $W_{n}\cup \qq$ is a \ka\ set in $\R^{d}$ for some $\gn$. If we choose $W_{n}=B(0,n)$ for every $n\ge 1$, and proceed as in the proof of Example \ref{Example 8.5}, we obtain that $B(0,\delta )\cup Q$ is a \ka\ set in $\R^{d}$ for every $\delta >0$.
\end{proof}
We now move out of the commutative setting, and present a characterization of \ka\ sets in the Heisenberg groups $H_{n}$.

\subsection{\ka\ sets in the Heisenberg groups $H_{n}$}\label{Subsection 8.2}
The Heisenberg group  of dimension $n\ge 1$, denoted by $H_{n}$, is formed of triples $(t,\qqq,\ppp)$ of $\R\times \R^{n}\times\R^{n}=\R^{2n+1}$. The group operation is given by
\[
(t_{1},\qqq_{1},\ppp_{1})\cdot (t_{2},\qqq_{2},\ppp_{2})=(t_{1}+t_{2}+\frac{1}{2}(
\ppp_{1}\cdot \qqq_{2}-\ppp_{2}\cdot\qqq_{1}),\qqq_{1}+\qqq_{2},\ppp_{1}+\ppp_{2}),
\]
where $\ppp\cdot\qqq$ denotes the scalar product of two vectors $\ppp$ and $\qqq$ of $\R^{n}$.
Irreducible unitary representations of $H_{n}$ are completely classified (see for instance \cite[Ch.~2]{T}, or \cite[Cor.~6.51]{Fo}): there are two distinct families of such representations, which we denote respectively by $(\mathcal{F}_{1})$ and $(\mathcal{F}_{2})$:
\par\smallskip 
-- the representations belonging to the family $(\mathcal{F}_{1})$ are representations of $H_{n}$ on $L^{2}(\R^{n})$. They are parametrized by an element of $\R$, which we write as $\pm\lambda $ with $\lambda >0$. Then 
$\pi _{\pm\lambda }(t,\qqq,\ppp)$, $(t,\qqq,\ppp)\in\R^{2n+1}$, acts on $L^{2}(\R^{n})$ as
\[
\xymatrix{
\pi _{\pm\lambda }(t,\qqq,\ppp)\,u\,:\,\xxx\ar@{|->}[r]&e^{i(\pm\lambda t\pm\sqrt{\lambda }\qqq\cdot \xxx+\frac{\lambda }{2}\qqq\cdot\ppp)}u(\xxx+\sqrt{\lambda }\,\ppp)
}
\]
where $u$ belongs to $L^{2}(\R^{n})$. 
These representations have the following important property, which will appear again in the next subsection:
\begin{fact}\label{Fact E}
 For every $\pm\lambda \in\R$ and every $u,v\in L^{2}(\R^{n})$, 
\[
\xymatrix{\pss{\pi _{\pm\lambda }(t,\qqq,\ppp)\,u}{v}\ar[r]&0}\quad\textrm{as}\quad
\xymatrix{|\ppp|\ar[r]&+\infty .}
\]
\end{fact}
\begin{proof}
 This follows directly from the dominated convergence theorem if $u$ and $v$ have compact support in $\R^{n}$. It then suffices to approximate $u$ and $v$ by functions with compact support to get the result.
\end{proof}
\par\smallskip 
-- the representations belonging to the family $(\mathcal{F}_{2})$ are one-dimensional. They are parametrized by elements $(\yyy,\eeeta )$ of $\R^{2n}$: for every $(t,\qqq,\ppp)\in H_{n}$,
\[
\pi _{\yyy,\eeeta }(t,\qqq, \ppp)=e^{i(\yyy\cdot\qqq+\eeeta\cdot \ppp)}.
\]
\par\smallskip 
We denote by $\pi_n$ the projection $(t,\qqq,\ppp)\longmapsto (\qqq,\ppp)$ of $H_{n}$ onto $\R^{2n}$.
Our main result concerning \ka\ sets in $H_{n}$ is the following:

\begin{theorem}\label{Proposition F}
 Let $\qq$ be a subset of the Heisenberg group $H_{n}$, $n\geq 1$. The following assertions are equivalent:
\begin{enumerate}
 \item [(1)] $\qq$ is a \ka\ set in $H_{n}$;
\item[(2)] $\pi_{n}(\qq)$ is a \ka\ set in $\R^{2n}$.
\end{enumerate}
\end{theorem}

\begin{proof} We set 
$\qq_0=\pi_{n}(\qq)$.
 The proof of Theorem \ref{Proposition F} relies on the same kind of ideas as those employed in the proof of Proposition \ref{Proposition D}. We start with the easy implication, which is that $(1)$ implies $(2)$. Suppose that $\qq$  is a \ka\ set in $H_{n}$, and let $\varepsilon >0$ be a \ka\ constant for $\qq$. 
%  In order to show that $\qq_{0}$ is a \ka\ set in $\R^{2n}$, we apply Theorem \ref{nimporte}: 
 Let $\sigma $ be a probability measure on $\R^{2n}$ such that
\begin{equation}\label{Equation 20}
 \sup_{(\qqq,\ppp)\in\qq_{0}}\,\bigl|\widehat{\sigma }(\qqq,\ppp)-1\bigr|=\sup_{(\qqq,\ppp)\in\qq_{0}}
\Bigl|\int_{\R^{2n}}e^{i(\yyy\cdot \qqq+\eeeta \cdot\ppp)}d\sigma (\yyy,\eeeta )-1\Bigr|<\dfrac{\varepsilon^{2}}{2} 
\end{equation}
and consider the representation $\rho $ of $H_{n}$ on $L^{2}(\R^{2n},\sigma )$ defined by 
\[
\xymatrix{
\rho (t,\qqq,\ppp)f\,:\,(\yyy,\eeeta )\ar@{|->}[r]&e^{i(\yyy\cdot\qqq+\eeeta \cdot\ppp)}f(\yyy,\eeeta )
}
\]
for every $(t,\qqq,\ppp)\in H_{n}$ and every $f\in L^{2}(\R^{2n},\sigma )$. Then (\ref{Equation 20}) implies that the constant function $\mathbbm{1}$ is a $(\qq,\varepsilon )$-invariant vector for $\rho $. Since $(\qq,\varepsilon )$ is a \ka\ pair in $H_{n}$, it follows that $\rho $ admits a non-zero $H_{n}$-invariant function $f\in L^{2}(\R^{2n},\sigma )$. Fix a representative of $f$, and consider the subset $A$ of $\R^{2n}$ consisting of pairs $(\yyy,\eeeta)$ such that $f(\yyy,\eeeta )\neq 0$. Then $\sigma (A)>0$, and for every $(\qqq,\ppp)\in\R^{2n}$, $\sigma $-almost every element $(\yyy,\eeeta )$ of $A$ satisfies $\yyy\cdot\qqq+\eeeta \cdot\ppp\in 2\pi \Z$.
Hence $\sigma $-almost every element $(\yyy,\eeeta )$ of $A$ has the property that $\yyy\cdot\qqq+\eeeta \cdot\ppp\in 2\pi \Z$
for every $(\qqq,\ppp)\in\Q^{2n}$. By continuity, $\sigma $-almost every element $(\yyy,\eeeta )$ of $A$ has the property that $\yyy\cdot\qqq+\eeeta \cdot\ppp$ belongs to $2\pi \Z$ for every $(\qqq,\ppp)\in\R^{2n}$, so that $(\yyy,\eeeta )=(\zzzero,\zzzero)$. We have thus proved that 
$\sigma (\{\zzzero,\zzzero\})>0$, and it follows that $\qq_{0}$ is a \ka\ set in $\R^{2n}$. 
\par\smallskip
Let us now prove the converse implication. Suppose that $\qq_{0}$ is a \ka\ set in $\R^{2n}$, and let $0<\varepsilon <3$ be a \ka\ constant for  $\qq_{0}$.  
 Let $\pi $ be a unitary representation of $H_{n}$ on a separable Hilbert space $H$, which admits a  $(\qq, \frac{\varepsilon}{8} )$-invariant vector $x\in H$ of norm $1$. We write as usual $\pi $ as a direct integral $\pi =\int_{Z}^{\oplus}\pi _{z}\,d\mu (z)$, where $\mu $ is a finite Borel measure on a standard Borel space $Z$, and $x$ as $(x_{z})_{z\in Z}$, with 
$\int_{Z}||x_{z}||^{2} d\mu (z)=1$. We have 
\begin{equation}\label{Eq17}
\sup_{(t,\qqq,\ppp)\in\qq}\,\Bigl|1-\int_{Z}\pss{\pi _{z}(t,\qqq,\ppp)x_{z}}{x_{z}}d\mu (z)\Bigr|<\dfrac{\varepsilon}{8}\cdot 
\end{equation}
For every $z\in Z$, the irreducible representation $\pi_{z}$ belongs to one of the two families $(\mathcal{F}_{1})$ and $(\mathcal{F}_{2})$. If $\pi _{z}$ belongs to $(\mathcal{F}_{1})$, we write it as $\pi _{\pm\lambda_{z} }$ for some $\pm\lambda _{z}\in\R$, and if $\pi $ belongs to $(\mathcal{F}_{2})$, as $\pi _{\yyy_{z},\eeeta _{z}}$ for some $(\yyy_{z},\eeeta _{z})\in\R^{2n}$. Let, for $i=1,2$, $Z_{i}$ be the subset of $Z$ consisting of the elements $z\in Z$ such that $\pi _{z}$ belongs to $(\mathcal{F}_{i})$. We have $x_{z}\in L^{2}(\R^{n})$ for every $z\in Z_{1}$, and $x_{z}\in\C$ for every $x\in Z_{2}$.
We now observe the following:
\begin{lemma}\label{Lemma G}
 A \ka\ subset of $\R^{2n}$ contains elements $(\qqq,\ppp)$ such that the Euclidean norm $|\ppp|$ of $\ppp$ is arbitrarily large.
\end{lemma}
\begin{proof}
 Let $\qq_{1}$ be a \ka\ subset of $\R^{2n}$, with \ka\ constant $\varepsilon >0$, and suppose that there exists a constant $M>0$ such that $|\ppp|\le M$ for every $(\qqq,\ppp)\in\qq_{1}$. Let $\delta >0$ be such that $2M\delta <\varepsilon $ and consider the probability measure on $\R^{2n}$ defined by 
\[
\sigma =\delta _{\0}\times\mathbbm{1}_{B(\0,\delta )}\,\dfrac{d\ppp}{|B(\0,\delta)| }\cdot 
\]
For every $(\qqq,\ppp)\in\qq_{1}$,
\[
|\widehat{\sigma }(\qqq,\ppp)-1|=\Bigl|\int_{B(\0,\delta )}e^{i\sss\cdot\ppp}\,\dfrac{d\sss}{|B(\0,\delta )| }-1\Bigr|\le 2\delta |\ppp|\le 2M\delta <\varepsilon. 
\]
But $\sigma (\{(\0,\0)\})=0$, and it follows  that $\qq_{1}$ is not a \ka\ set in $\R^{2n}$, which is a contradiction.
\end{proof}

By Fact \ref{Fact E}, we have for every $x\in Z_{1}$
\[
\xymatrix{
\pss{\pi _{\pm\lambda_{z} }(t,\qqq,\ppp)x_{z}}{x_{z}}\ar[r]&0
}\quad\textrm{as}\ \xymatrix@C=10pt{|\ppp|\ar[r]&+\infty },\ (t,\qqq,\ppp)\in H_{n}.
\]
Since $|\pss{\pi _{\pm\lambda }(t,\qqq,\ppp)x_{z}}{x_{z}}|\le||x_{z}||^{2}$ for every $z\in Z_{1}$, and 
$\int_{Z}||x_{z}||^{2}d\mu (z)=1$, the dominated convergence theorem implies that
\[
\xymatrix{\ds\int_{Z_{1}}
\pss{\pi _{\pm\lambda_{z} }(t,\qqq,\ppp)x_{z}}{x_{z}}d\mu (z)\ar[r]&0
}\quad\textrm{as}\ \xymatrix@C=10pt{|\ppp|\ar[r]&+\infty },\ (t,\qqq,\ppp)\in H_{n}.
\]
By Lemma \ref{Lemma G}, there exists an element $(t_{0},\qqq_{0},\ppp_{0})$ of $\qq$ with $|\ppp_{0}|$ so large that 
\[
\Bigl|\int_{Z_{1}}
\pss{\pi _{\pm\lambda_{z} }(t_{0},\qqq_{0},\ppp_{0})x_{z}}{x_{z}}d\mu (z)\Bigr|<\dfrac{\varepsilon }{8}\cdot 
\]
Property (\ref{Eq17}) implies then that
\[
\Bigl|1-\int_{Z_{2}}\pi _{\yyy_{z},\eeeta _{z}}(t_{0},\qqq_{0},\ppp_{0})|x_{z}|^{2}d\mu (z)\Bigr|<\dfrac{\varepsilon }{4}
\] 
from which it follows that $\ds\int_{Z_{2}}|x_{z}|^{2}d\mu (z)>1-\dfrac{\varepsilon }{4}$, so that
$\ds\int_{Z_{1}}||x_{z}||^{2}d\mu (z)<\dfrac{\varepsilon }{4}$.  Plugging this into (\ref{Eq17}) yields that
\[
\sup_{(t,\qqq,\ppp)\in\qq}\,\Bigl|1-\int_{Z_{2}}\pi _{\yyy_{z},\eeeta _{z}}(t,\qqq,\ppp)|x_{z}|^{2}d\mu (z)\Bigr|<\dfrac{3\varepsilon }{8}\cdot 
\]
Since $\int_{Z_{2}}|x_{z}|^{2}d\mu (z)>1-\varepsilon /4$ and $0<\varepsilon <3$, we can, by normalizing the family $(x_{z})_{z\in Z_{2}}$, suppose without loss of generality that $Z=Z_{2}$,  $\int_{Z}|x_{z}|^{2}d\mu (z)=1$ and that
\begin{equation}\label{Eq18}
 \sup_{(t,\qqq,\ppp)\in\qq}\,\Bigl|1-\int_{Z}e^{i(\yyy_{z}\cdot\qqq+\eeeta _{z}\cdot\ppp)}\,|x_{z}|^{2\,}d\mu (z)\Bigr|<\varepsilon . 
\end{equation}
 Consider now the unitary representation $\rho $ of $\R^{2n}$ on $L^{2}(Z,\mu )$ defined by 
\[
\xymatrix{
\rho (\qqq,\ppp)\,f\,:\,z\ar@{|->}[r]&e^{i(\yyy_{z}\cdot\qqq+\eeeta _{z}\cdot\ppp)}f_z
}
\]
for every $(\qqq,\ppp)\in\R^{2n}$ and every $f=(f_z)_{z\in Z}\in L^{2}(Z,\mu )$. Then (\ref{Eq18}) can be 
rewritten as 
\begin{align*}
 \sup_{(t,\qqq,\ppp)\in\qq}\,\bigl|1-\pss{\rho (\qqq,\ppp)x}{x}\bigr|<\varepsilon,\quad\textrm{i.e.}\quad
\sup_{(\qqq,\ppp)\in\qq_{0}}\,\bigl|1-\pss{\rho (\qqq,\ppp)x}{x}\bigr|<\varepsilon.
\end{align*}
Since $\varepsilon $ is a \ka\ constant for $\qq_{0}$, the representation $\rho $ admits a non-zero $\R^{2n}$-invariant vector $f\in L^{2}(Z,\mu )$. Proceeding as in the proof of $(1)\Longrightarrow (2)$, we see that for every $(\qqq,\ppp)\in\R^{2n}$, $e^{i(\yyy_{z}\cdot\qqq+\eeeta _{z}\cdot\ppp)}f(z)=f(z)$ $\mu $-almost everywhere on $Z$, so that there exists a subset $Z_{0}$ of $Z$ with $\mu (Z_{0})>0$
such that $f$ does not vanish on $Z_{0}$ and,  for every $z\in Z_{0}$, $\yyy_{z}\cdot\qqq+\eeeta _{z}\cdot\ppp$ belongs to $2\pi \Z$ for every $(\qqq,\ppp)\in \Q^{2n}$. By continuity, $\yyy_{z}\cdot\qqq+\eeeta _{z}\cdot\ppp$ belongs to $2\pi \Z$ for every $z\in Z_{0}$ and every $(\qqq,\ppp)\in\R^{2n}$, so that $(\yyy_{z},\eeeta _{z})=(\zzzero,\zzzero)$ for every $z\in Z_{0}$. So if we set $Z_{0}'=\{z\in Z\,;\,(\yyy_{z},\eeeta _{z})=(\zzzero,\zzzero)\}$, we have $\mu (Z_{0}')>0$. The function 
$f=\mathbbm{1}_{Z'_{0}}$ is hence a non-zero element of $L^{2}(Z,\mu )$, which is clearly an $H_{n}$-invariant vector for the representation $\pi $. So $(\qq,\frac{\varepsilon }{8})$ is a \ka\ pair in $H_{n}$, and Theorem \ref{Proposition F} is proved.
\end{proof}

\subsection{\ka\ sets in the  group $\textrm{Aff}_+(\R)$}
The underlying space of the  group $\textrm{Aff}_+(\R)$ of orientation-preserving affine homeomorphisms of $\R$ is $(0,+\infty )\times \R$, and the group law is given by 
$(a,b)(a',b')=(aa',b+ab')$, where $(a,b)$ and $(a',b')$ belong to $ (0,+\infty )\times \R$. 
% In this section we denote this group by $G$. 
As in the case of the Heisenberg groups, the irreducible unitary representations of $\textrm{Aff}_+(\R)$ are completely classified (see \cite[Sec.~6.7]{Fo}) and fall within two classes:
\par\smallskip 
-- the class $(\mathcal{F}_{1})$ consists of two infinite dimensional representations $\pi _{+}$ and $\pi _{-}$ of $\textrm{Aff}_+(\R)$, which act respectively on the Hilbert spaces $L^{2}((0,+\infty ),ds)$ and $L^{2}((-\infty ,0),ds)$. They are both defined by the formula
\[
\xymatrix{
\pi _{\pm}(a,b)f\,:\,s\ar@{|->}[r]&\sqrt{a}\,e^{2i\pi bs}f(as)
}
\] where $(a,b)\in (0,+\infty )\times \R$, $f\in L^{2}((0,+\infty ),ds)$ in the case of $\pi _{+}$, and  $f\in L^{2}((-\infty ,0),ds)$ in the case of $\pi _{-}$. It is a direct consequence of the Riemann-Lebesgue lemma that the analogue of Fact \ref{Fact E} holds true for the two representations $\pi _{+}$ and $\pi _{-}$ of $G$:
\begin{fact}\label{Fact M}
 For every $f_1,f_2\in L^{2}((0,+\infty ),ds)$ and every $g_1,g_2\in L^{2}((-\infty ,0),ds)$, we have
\[
\xymatrix{\pss{\pi _{+}(a,b)f_1}{f_2}\ar[r]&0}\quad \textrm{and}\quad \xymatrix{\pss{\pi _{-}(a,b)g_1}{g_2}\ar[r]&0}\quad\textrm{as}\quad
\xymatrix{|b|\ar[r]&+\infty .}
\]
\end{fact}
\par\smallskip 
-- the representations of $\textrm{Aff}_+(\R)$ belonging to the family $(\mathcal{F}_{2})$ are one-dimensional. They are parametrized by $\R$, and $\pi _{\lambda }$ is defined for every $\lambda \in\R$ by the formula
\[
\pi _{\lambda }(a,b)=a^{i\lambda }\quad\textrm{for every}\ (a,b)\in (0,+\infty )\times \R.
\]
Proceeding as in the proof of Theorem \ref{Proposition F}, we characterize the \ka\ subsets of the  group $\textrm{Aff}_+(\R)$ in the following way:
\begin{theorem}\label{Proposition N}
 Let $\qq$ be a subset of $\textrm{Aff}_+(\R)$. The following assertions are equivalent:
\begin{enumerate}
 \item [{(1)}] $\qq$ is a \ka\ set in $\textrm{Aff}_+(\R)$;
\item[{(2)}] the set $\qq_{0}=\{t\in\R\,;\,\exists\,b\in\R\ (e^{t},b)\in\qq\}$ is a \ka\ set in $\R$.
\end{enumerate}
\end{theorem}

\begin{proof}
 The proof is similar to that of Theorem \ref{Proposition F}, and we will not give it in full detail here. Let us first sketch briefly a proof of the implication $(1)\Longrightarrow(2)$. Suppose that $\qq$ is a \ka\ set in $\textrm{Aff}_+(\R)$, and let $\varepsilon >0$ be a \ka\ constant for $\qq$. Consider a probability measure $\sigma $ on $\R$ such that 
$
 \sup_{t\in\qq_{0}}\,\bigl|\widehat{\sigma }(t)-1\bigr|<{\varepsilon ^{2}}/{2}. 
$
We associate to $\sigma $ a representation $\rho $ of $\textrm{Aff}_+(\R)$ on $L^{2}(\R,\sigma )$ by setting, for every $(a,b)\in (0,+\infty )\times \R$ and every $f\in L^{2}(\R,\sigma )$,
$\xymatrix{
\rho (a,b)f\,:\,s\ar@{|->}[r]&e^{i s (\ln a)}f(s).
}$ Since
\[
||\rho (a,b)\mathbbm{1}-\mathbbm{1}||^{2}\le2\Bigl|\int_{\R}\bigl(e^{is(\ln a)}-1)d\sigma (s)\Bigr|\quad \textrm{for every}\ (a,b)\in (0,+\infty )\times \R,
\]
we have $\sup_{\{(a,b)\,;\,\ln a\in\qq_{0}\}}\,||\rho (a,b)\mathbbm{1}-\mathbbm{1}||<\varepsilon $, i.\,e.\ $\sup_{(a,b)\in\qq}\,||\rho (a,b)\mathbbm{1}-\mathbbm{1}||<\varepsilon $. Hence $\rho $ admits a non-zero $\textrm{Aff}_+(\R)$-invariant function $f\in L^{2}(\R,\sigma )$, and the same argument as in the proof of Theorem \ref{Proposition F} shows then that $\sigma (\{0\})>0$. The converse implication 
$(2)\Longrightarrow(1)$ is proved in exactly the same way as in Theorem \ref{Proposition F}, using the same modifications as those outlined above.
The group $\R^{2n}$ has to be replaced by the multiplicative group $((0,+\infty ),\times)$ and the analogue of Lemma \ref{Lemma G} is that \ka\ subsets of this group contain elements of arbitrarily large absolute value. If $\qq_{0}$ is a \ka\ set in $\R$, with \ka\ constant $\varepsilon $ small enough, the same argument as in the proof of Theorem \ref{Proposition F} (involving the same notation) shows that it suffices to prove the following statement: let $\mu $ be a finite Borel measure on a Borel space $Z$, $x=(x_{z})_{z\in Z}$ a scalar-valued function of $L^{2}(Z,\mu )$ with $\int_{Z}|x_{z}|^{2}d\mu (z)=1$, and $\pi $ a representation of $G$ of the form $\pi =\int_{Z}^{\oplus}\pi _{\lambda _{z}}d\mu (z)$ with
\[
\sup_{(a,b)\in\qq}\bigl|1-\pss{\pi(a,b)x}{x}\bigr|=\sup_{\{a\,;\,\ln  a\,\in\,\qq_{0}\}}\Bigl|1-\int_{Z}e^{i(\ln a)\lambda _{z}}|x_{z}|^{2}d\mu (z)\Bigr|<\varepsilon.
\]
Then the set $Z_{0}=\{z\in Z\,;\,\lambda _{z}=0\}$ satisfies $\mu (Z_{0})>0$. The proof of this statement uses the same argument as the one employed in the proof of Theorem \ref{Proposition F}. It involves the representation $\rho $ of the group $((0,+\infty ),\times)$ on $L^{2}(Z,\mu )$ defined by 
$\smash{\xymatrix{\rho (a)f:z\ar@{|->}[r]&e^{i(\ln a)\lambda _{z}}f_{z}}}$ for every $a>0$ and every $(f_{z})_{z\in Z}\in L^{2}(Z,\mu )$, and uses the obvious fact that since $\qq_{0}$ is a \ka\ set in $\R$, $\{a\,;\,\ln a\,\in\,\qq_{0}\}$ is a \ka\ set in $((0,+\infty ),\times)$.
\end{proof}

\begin{remark}\label{Howe-Moore}
Facts \ref{Fact E} and \ref{Fact M} have played a crucial role in the proofs of Theorems \ref{Proposition F} and \ref{Proposition N} respectively, as they allowed us to discard all irreducible representations except the one-dimensional ones in  inequalities of the form (\ref{Eq17}).
In groups with the Howe-Moore property (see for instance \cite{HM}, \cite{Z} or  \cite{BM} for the definition and for more about this property), all non-trivial irreducible representations have the vanishing property of the matrix coefficients stated in Facts \ref{Fact E} or \ref{Fact M}. 
It easily follows from this observation that all subsets with non-compact closure are \ka\ sets in groups with the Howe-Moore property, and that if the group is additionally supposed not to have Property (T), the \ka\ sets are exactly the sets with non-compact closure.
As $SL_{2}(\R)$ is a non-compact connected simple real Lie group with finite center, it has the Howe-Moore property. But it does not have Property (T), and so we have:
\begin{example}\label{Example J}
 The \ka\ sets in $SL_{2}(\R)$ are exactly the subsets of $SL_{2}(\R)$ with non-compact closure.
 \end{example}
 
These observations testify of the rigidity of the structure of groups with the Howe-Moore property, and stand in sharp contrast with all the examples we have presented in the rest of this section.
\end{remark}

 \appendix
\section{Infinite tensor products of Hilbert spaces}
We  briefly describe in this appendix some constructions of tensor 
products of infinite families of Hilbert spaces, and of tensor products of infinite families of unitary representations. These last objects play an important role in the proof of Theorem \ref{Theorem 0}. We review here the properties and results which we need, following the original works of von Neumann \cite{VN} and Guichardet \cite{Gui}.
\subsection{The complete and incomplete tensor products of Hilbert 
spaces}
The original construction of the complete and incomplete tensor products 
of a family $(H_{\alpha })_{\alpha \in I}$ of Hilbert spaces is due to von 
Neumann \cite{VN}. It was later on taken up by Guichardet in \cite{Gui} 
under a somewhat different point of view, and the incomplete tensor 
products of von Neumann are rather known today as the Guichardet tensor 
products of Hilbert spaces. Although these constructions can be carried 
out starting from an arbitrary  family $(H_{\alpha })_{\alpha \in I}$ of 
Hilbert spaces, we will present them here only in the case of a countable 
family $(\hn)_{\gn}$ of (complex) Hilbert spaces.
\par\smallskip 
The \emph{complete infinite tensor product} $\inc{}$ 
of the Hilbert spaces $\hn $ is defined in \cite[Part~II,~Ch.~3]{VN} in 
the following way: the elementary infinite tensor products are the elements
$\xxx=\oxn$, where $\xn$ belongs to $\hn $ for each 
$\gn$  and the infinite product $\prod_{\gn}||\xn||$ is convergent 
in the sense of \cite[Def.~2.2.1]{VN}, which by \cite[Lem.~2.4.1]{VN} is equivalent to the fact that either $x_{n}=0$ for some $n\ge 1$ or the series $\sum_{n\ge 1}\max(||x_{n}||-1,0)$ is convergent. Sequences 
$\pxn$ with this property are called by von Neumann in 
\cite{VN} \emph{$C$-sequences}. A scalar product is then defined on the 
set of finite linear combinations of elementary tensor products by setting 
\[
\pss{\xxx}{\yyy}=\prod_{\gn}\pss{\xn}{\yn}
\]
for any elementary tensor products $\xxx=\oxn$ and 
$\yyy=\oyn$, and extending the definition by linearity to 
finite linear combinations of such elements.
The product defining $\pss{\xxx}{\yyy}$ for two elementary vectors $\xxx$ and $\yyy$ is quasi-convergent in the sense of \cite[Def.~2.5.1]{VN}, i.e. $\prod_{\gn}|\pss{\xn}{\yn}|$ is convergent. The value of this quasi-convergent product is $\prod_{\gn}\pss{\xn}{\yn}$ if the product is convergent in the usual sense, and $0$ if it is not.
\par\smallskip 
That this is indeed a scalar product which turns the set of finite linear 
combinations of elementary tensor products into a complex prehilbertian 
space is proved in \cite[Lem.~3.21 and Theorem II]{VN}. 
For any elementary tensor product $\xxx=\oxn$, 
$||\xxx||=\prod_{\gn}||\xn||$. The space $\inc{}$ is 
the completion of this space for the topology induced  by the scalar 
product. It is always non-separable.
\par\smallskip 
The \emph{incomplete tensor products} are closed subspaces of the complete 
tensor product. They are defined by von Neumann using an equivalence 
relation between sequences $\pxn$ of vectors with $\xn\in
\hn $ for each $\gn$ and such that the series 
$\sum_{\gn}\bigl|1-||\xn||\bigr|$ is convergent. Such sequences are called 
\emph{$C_{0}$-sequences}. They are $C$-sequences, and if 
$\pxn$ is a $C$-sequence such that $\prod_{\gn}||\xn||>0$ (i.\,e.\ 
$\pxn$ is non-zero in $\inc{}$) then 
$\pxn$ is a $C_{0}$-sequence. If $\pxn$ is a $C_{0}$-sequence, $\pxn$ is 
bounded, and the series $\sum_{\gn}\bigl|1-||\xn||^{2}\bigr|$ is 
convergent.
\par\smallskip 
Two $C_{0}$-sequences $\pxn$ and $\pyn$ are \emph{equivalent} if the 
series $\sum_{\gn}\bigl|1-\pss{\xn}{\yn}\bigr|$ is convergent. If 
$\AAA$ denotes an equivalence class of $C_{0}$-sequences for this 
equivalence relation, the \emph{incomplete tensor product} 
$\inc{\AAA}$ \emph{associated to} $\AAA$ is the 
closed linear span in $\inc{}$ of the vectors $\xxx=
\oxn$, where $\pxn$ belongs to $\AAA$ \cite[Def.~4.1.1]{VN}. 
If $\AAA$ and $\AAA'$ are two different equivalence classes, the spaces 
$\smash[t]{\inc{\AAA}}$ and $\smash[t]{\inc{\AAA'}}$ are orthogonal, and 
the linear span of the 
incomplete tensor products $\inc{\AAA}$, where $\AAA$ runs over all 
equivalence classes of $C_{0}$-sequences, is dense in the complete tensor 
product $\inc{}$.
\par\smallskip 
If $\AAA$ is an equivalence class of $C_{0}$-sequences, $\inc{\AAA}$ 
admits another, more transparent description, which runs as follows 
\cite[Lem.~4.1.2]{VN}, see also \cite[Rem.~1.1]{Gui}: let $\pan$ be a 
sequence with $\an\in \hn $ and $||\an||=1$ for every $\gn$, such that the equivalence class of 
$\pan$ is $\AAA$ (such a sequence $\pan $ does exist: if $\pxn $ is any non-zero $C_{0}$-sequence belonging to $\AAA$, $x_{n}$ is non-zero for every $n\ge 1$, and we can define a $C_{0}$-sequence $\pan $ by setting $a_{n}=x_{n}/||x_{n}||$ for every $n\ge 1$. It is not difficult to check that $\pan $ is equivalent to $\pxn $, and so belongs to $\AAA$). Then $\inc{\AAA}$ coincides with the closed linear span 
in $\inc{}$ of vectors $\xxx=\oxn$, where $\xn=\an$ for all 
but finitely many  integers $\gn$. Denoting the vector $\oan$ by $\aaaa$, 
we write this closed linear span as $\inc{\aaaa}$ (see \cite{Gui}), and 
thus $\inc{\aaaa}=\inc{\AAA}$, where $\AAA$ is the equivalence class of 
$\aaaa$. The space $\inc{\aaaa}$ is usually called \emph{the Guichardet 
tensor product} of the spaces $\hn $ \emph{associated to the sequence} 
$\pan$. Proposition 1.1 of \cite{Gui} states the following, which is a direct consequence of
the discussion above: if $\xxx=\pxn$ is a $C_{0}$-sequence 
which is equivalent to $\aaaa$, $\xxx$ belongs to $\inc{\aaaa}$. Vectors 
$\xxx$ of this form are also said to be \emph{decomposable with respect 
to} $\aaaa$, while vectors $\xxx=\pxn$ with $\xn=\an$ for all but finitely 
many indices $n$ are called \emph{elementary vectors} of $\inc{\aaaa}$.
\par\smallskip
Suppose that all the spaces $\hn$, $\gn$, are separable.
For each $\gn$, let $(e_{p,n})_{1\le p\le p_{n}}$ be a Hilbertian basis of 
$\hn $, with $1\le p_{n}\le+\infty $ and $e_{1,n}=a_{n}$. The family of all
elementary vectors $\eee_{\beta }=\oti_{n\ge 1}e_{\beta (n),n}$ of $\inc{\aaaa}$, where 
$\beta $ is a map from $\N$ into itself such that $1\le\beta (n)\le p_{n}$ 
for every $\gn$ and $\beta (n)=1$ for all but finitely many integers $n\ge 
1$, forms a Hilbertian basis of $\inc{\aaaa}$ \cite[Lem.~4.1.4]{VN}. In 
particular, $\inc{\aaaa}$ is a separable complex Hilbert space.

\subsection{Tensor products of unitary representations} 
Let $G$ be a topological group, and let $(\hn)_{\gn}$ be a sequence of 
complex separable Hilbert spaces. Let $\pan$ be a sequence of 
vectors with $\an\in \hn $ and $||\an||=1$ for every $\gn$. We are 
looking 
for conditions under which one can define a unitary representation 
$\pmb{\pi} $ 
of $G$ on $\inc{\aaaa}$ which satisfies 
\begin{equation}\label{Eq1}
 \pmb{\pi} (g)\oxn=\oti_{\gn}\pi_{n} (g)\xn
\end{equation}
for every $g\in G$ and every decomposable vector $\xxx=\oxn$ with respect 
to $\aaaa$. 
Observe that without any assumption, the equality $\pmb{\pi} 
(g)\oxn=\oti_{\gn}\pi_{n} (g)\xn$ does not make any sense, since 
$(\pi_{n} (g)\xn)_{\gn}$, which is a $C_{0}$-sequence, may not be 
equivalent 
to $\aaaa$, and thus may not belong to $\inc{\aaaa}$.
\par\smallskip 
Infinite tensor products of unitary representations have already been 
studied in various contexts (see for instance \cite{BC} and the 
references therein). In   
\cite[Prop.~2.3]{BC}, the following observation is made: suppose that, 
for each $\gn$, $U_{n}$ is a unitary 
operator on $\hn$. Then there exists a unitary operator 
$\pmb{U}=\oti_{\gn}U_{n}$ on $\inc{\aaaa}$ satisfying 
\[\pmb{U}\bigl(\oxn \bigr)=\oti_{\gn}U_{n}\xn\]
for every decomposable vector $\xxx=\oxn$ with respect to $\aaaa$ if and only if the series 
$\sum_{\gn}\bigl|1-\pss{U_{n}\an}{\an} \bigr|$ is convergent (which is 
equivalent to requiring that the $C_{0}$-sequence $(U_{n}\an)_{\gn}$ be 
equivalent to $\pan$, i.\,e.\ to the fact that $\oti_{\gn}U_{n}a_{n}$  be a 
decomposable vector with respect to $\aaaa$). It follows from  this 
result that the formula (\ref{Eq1}) makes sense in $\inc{\aaaa}$ if and 
only if the series 
\begin{equation}\label{Eq2}
 \sum_{\gn}\,\bigl|1-\pss{\pi _{n}(g)a_{n}}{a_{n}}\bigr|
\end{equation}
is convergent for every $g\in G$.
Under this condition $\pmb{\pi} (g)=\oti_{\gn}\pi _{n}(g)$ is a unitary 
operator on $\inc{\aaaa}$ for every 
$g\in 
G$, and $\pmb{\pi} (gh)=\pmb{\pi} 
(g)\,\pmb{\pi} (h)$ for 
every $g,h\in G$.
\par\smallskip 
If the group $G$ is discrete, this tensor product representation is of 
course automatically strongly continuous. It is also the case if $G$ is 
supposed to be locally compact. 
\begin{proposition}\label{Proposition 3.2.0}
 Suppose that $G$ is a  locally compact group, and that the series 
$
 \sum_{\gn}|1-\pss{\pi 
 _{n}(g)\an }{\an}|
$ 
 is convergent for every $g\in G$. Then $\pmb{\pi} 
=\oti_{\gn}\pi _{n}$ is strongly continuous, and is hence a unitary 
representation of $G$ on $\inc{\aaaa}$.
\end{proposition}
\begin{proof}
 Since all the spaces $\hn $, $\gn$, are separable, $\inc{\aaaa}$ is 
separable too, and by \cite[Lem.~A.6.2]{BdHV} it suffices to show that 
$\smash{\xymatrix{g\ar@{|->}[r]&\pss{\pmb{\pi} (g)\,\pmb{\xi} }{\pmb{\xi} 
}}}$ 
is a 
measurable 
map from $G$ into $\C$ for every vector $\pmb{\xi} \in\inc{\aaaa}$.
% (measurability 
% refers to the Haar measure on $G$). 
Since the linear span of the elementary vectors is dense in
$\inc{\aaaa}$, standard arguments show that it 
suffices to prove this for elementary vectors $\xxx=\oti_{\gn}\xn$ of 
$\inc{\aaaa}$. 
Since 
each  map
$\smash{\xymatrix{g\ar@{|->}[r]&\pss{\pi_{n} (g)\xn }{\xn }}}$ is 
continuous 
on 
$G$, it is clear that $\smash{\xymatrix{\!\!g\ar@{|->}[r]&\pss{\pmb{\pi} 
(g)\xxx 
}{\xxx }}
=\prod_{\gn}\pss{\pi _{n}(g)\xn}{\xn}}$ is measurable on $G$.
\end{proof}
In the general case one 
needs 
to impose  an additional condition on the representations $\pi _{n}$ and on the vectors $a_{n}$ in order that $\pmb{\pi} $ be a strongly 
continuous representation of $G$ on $\inc{\aaaa}$.
\begin{proposition}\label{Proposition 3.2} 
 Suppose that the series $\sum_{\gn}\,\bigl|1-\pss{\pi 
_{n}(g)a_{n}}{a_{n}}\bigr|$ is convergent for every $g\in G$ and that the 
function 
$\xymatrix{g\ar@{|->}[r]&\sum_{\gn}\,\bigl|1-\pss{\pi 
_{n}(g)a_{n}}{a_{n}}\bigr|}$
is continuous on a neighborhood of the 
identity element $e$ of $G$. Then $\pmb{\pi} =\oti_{\gn}\pi _{n}$ is 
strongly 
continuous, and is hence a unitary representation of $G$ on $\inc{\aaaa}$.
\end{proposition}

\begin{proof}[Proof of Proposition \ref{Proposition 3.2}]
Since the linear span of the elementary vectors is dense in $\inc{\aaaa}$, 
and the operators $\pmb{\pi }(g)$, $g\in G$, are unitary, it suffices to 
prove that the map 
${\smash{\xymatrix{g\ar@{|->}[r]&\pmb{\pi}(g)\xxx}}}$ is continuous at $e$ 
for every elementary vector $\xxx=\oti_{\gn}\xn$ of norm $1$ of 
$\inc{\aaaa}$. Let $N\ge 1$ be such that $\xn=\an$ for every 
$n>N$. We have for every $g\in G$:
\begin{align*}
 ||\pmb{\pi }(g)\xxx-\xxx||^{2}&=2(1-\textrm{Re}\pss{\pmb{\pi }(g)\xxx}{\xxx})=
 2\biggl(1-\prod_{\gn}\textrm{Re}\pss{\pi 
_{n}(g)\dfrac{\xn}{||\xn||}}{\dfrac{\xn}{||\xn||}}\biggr)
\intertext{since $||\xxx||=\ds\prod_{n\ge 1}{||\xn||}$=1. Thus}
||\pmb{\pi }(g)\xxx-\xxx||^{2}&\le 2\sum_{\gn}\,\Bigl| 1-\pss{\pi 
_{n}(g)\dfrac{\xn}{||\xn||}}{\dfrac{\xn}{||\xn||}}\Bigr|\\
&\le 2\sum_{n=1}^{N}\,\Bigl| 1-\pss{\pi 
_{n}(g)\dfrac{\xn}{||\xn||}}{\dfrac{\xn}{||\xn||}}\Bigr|
+2\sum_{\gn}|1-\pss{\pi _{n}(g)\an}{\an}|.
\end{align*}
If $\varepsilon $ is any positive number, it follows from the assumptions that
$||\pmb{\pi }(g)\xxx-\xxx||<\varepsilon$ if $g$ 
lies in a suitable neighborhood of $e$. This proves the continuity of the 
map 
$\smash{\xymatrix{g\ar@{|->}[r]&\pmb{\pi }(g)\xxx.}}$
\end{proof}

\par\smallskip 
We finish this appendix by giving a sufficient condition for an infinite 
tensor product representation on a space $\inc{\aaaa}$ to be weakly 
mixing: let, for each $\gn$, $\hn$ be a separable Hilbert space, $\an$ a 
vector of $\hn$ with $||\an||=1$, and 
$\pi _{n}$ a unitary representation of $G$ on $H_{n}$. 
We suppose that the assumptions of either Proposition \ref{Proposition 3.2.0} (when $G$ is locally compact) or Proposition \ref{Proposition 3.2} (in the general case) are 
satisfied, so that $\pmb{\pi} =\oti_{\gn}\pi _{n}$ is a unitary 
representation of $G$
on $\inc{\aaaa}$. Then
\begin{proposition}\label{Proposition 4.3}
 In the case where $\underline{\lim}_{\,n\to+\infty }m(|\pss{\pi 
_{n}(\,\centerdot\,)\an}{\an}|^{2})=0,$ the representation $\pmb{\pi} 
=\oti_{\gn}\pi 
_{n} $ is weakly mixing.
\end{proposition}
\begin{proof}
 The proof of Proposition \ref{Proposition 4.3} relies on the same idea as 
that of Proposition \ref{Proposition 3.2}: let $\xxx=\oti_{\gn}\xn$ and 
$\yyy=\oti_{\gn}\yn$ be two elementary vectors in $\inc{\aaaa}$ with
$||\xxx||=||\yyy||=1$. We have 
\begin{align*}
 |\pss{\pmb{\pi} (g)\,\xxx}{\yyy}|^{2}&=\prod_{k\ge 1}\,\Bigl|\Bigl\langle 
\pi 
_{k}(g)\,\dfrac{x_{k}}{||x_{k}||},\dfrac{y_{k}}{||y_{k}||}\Bigr\rangle 
\Bigr|^{2}\le \Bigl|\Bigl\langle \pi 
_{n}(g)\,\dfrac{\xn}{||\xn||},\dfrac{y_{n}}{||y_{n}||}\Bigr\rangle 
\Bigr|^{2}
\intertext{for every $n\ge 1$ and every $g\in G$. But}
\Bigl|\Bigl\langle \pi 
_{n}(g)\,\dfrac{\xn}{||\xn||},\dfrac{\yn}{||\yn||}\Bigr\rangle 
\Bigr|
&\le|\pss{\pi _{n}(g)\an}{\an}|+\biggl| \biggl| 
\dfrac{\xn}{||\xn||}-\an\biggr| \biggr|+\biggl| \biggl|
\dfrac{\yn}{||\yn||}-\an\biggr| \biggr|.\\
\intertext{Squaring and taking the mean on both sides we obtain that}
m(|\pss{\pmb{\pi} (\,\centerdot\,)\xxx}{\yyy}|^{2})&\le 4\,m(|\pss{\pi_{n} 
(\,\centerdot\,)\,\an}{\an}|^{2})+4\,\biggl| \biggl| 
\dfrac{\xn}{||\xn||}-\an\biggr| \biggr|^{2}+4\,\biggl| \biggl|
\dfrac{\yn}{||\yn||}-\an\biggr| \biggr|^{2}
\end{align*}
for every $n\ge 1$. Since 
$\underline{\lim}_{\,n\to+\infty 
}m(|\pss{\pi 
_{n}(\,\centerdot\,)\an}{\an}|^{2})=0$ and the two other terms are equal to 
zero for $n$ sufficiently large, $m(|\pss{\pmb{\pi} 
(\,\centerdot\,)\xxx}{\yyy}|^{2})=0$. 
Weak mixing of $\pmb{\pi} $ now follows from standard density arguments.
\end{proof}


\begin{thebibliography}{99999}
{

\bibitem{AnBi} \textsc{M.~Anoussis, A.~Bisbas,}
\newblock
Continuous measures on compact Lie groups,
\newblock \emph{Ann. Inst. Fourier} {\bf 50} (2000), p.\,1277--1296. 

\bibitem{Ballo} \textsc{M.~Ballotti,}
\newblock
 Convergence rates for Wiener's theorem for contraction semigroups,
\newblock \emph{Houston J. Math.} {\bf 11} (1985), p.\,435--445. 

\bibitem{BalloGold} \textsc{M.~Ballotti, J.~Goldstein,}
\newblock
 Wiener's theorem and semigroups of operators, in:
\newblock Infinite dimensional systems (Retzhof, 1983), 
\emph{Lecture Notes in Math.} {\bf 1076} (1984), Springer, p.\,16--22.

\bibitem{BC} \textsc{E.~Bedos, R.~Conti,}
\newblock On infinite tensor products of projective unitary 
representations,
\newblock \emph{Rocky Mountain J. Math.} {\bf 34} (2004), p.\,467--493.


\bibitem{BV} \textsc{B.~Bekka, A.~Valette,}
\newblock \ka's Property (T) and amenable representations, 
\newblock \emph{Math. Zeit.} {\bf 212} (1993), p.\,293--299.

\bibitem{BdHV}  \textsc{B.~Bekka, P.~de~la~Harpe, A.~Valette,}
\newblock  \ka's Property (T),
\newblock \emph{New Mathematical Monographs} {\bf 11} (2008), Cambridge 
University Press.

\bibitem{BM}
\textsc{B.~Bekka, M.~Mayer,}
\newblock Ergodic theory and topological dynamics of group actions on homogeneous spaces,
\newblock \emph{London Mathematical Society Lecture Note Series} \textbf{269} (2000), Cambridge University Press.


\bibitem{BR} \textsc{V.~Bergelson, J.~Rosenblatt,} 
\newblock Mixing actions of groups,
\newblock \emph{Illinois J. Math.} {\bf 32} (1998), p.\,65--80.

\bibitem{BjFi} \textsc{M.~Bj\"orklund, A.~Fish,}
\newblock Continuous measures on homogenous spaces,
\newblock \emph{Ann. Inst. Fourier} {\bf 59} (2009), p.\,2169--2174.  

\bibitem{Bur} \textsc{R.~Burckel,}
\newblock Weakly almost periodic functions on semigroups, Gordon and 
Breach (1970).




\bibitem{DT} 
\textsc{M.~Drmota, R.~Tichy},
\newblock Sequences, discrepancies and applications,
\newblock \emph{Lecture Notes in Mathematics} \textbf{1651} (1997), Springer-Verlag. 

\bibitem{Dye} \textsc{H.~Dye,} 
\newblock On the ergodic mixing theorem,
\newblock \emph{Trans. Amer. Math. Soc.} {\bf 118} (1965), p.\,123--130.


\bibitem{EG}
\textsc{T.~Eisner, S.~Grivaux,}
\newblock Hilbertian Jamison sequences and rigid dynamical systems,
\newblock \emph{J. Funct. Anal.} \textbf{261} (2011), p.\,2013--2052. 

\bibitem{Farkas} \textsc{B.~Farkas,}
\newblock Wiener's lemma and the Jacobs-de Leeuw-Glicksberg decomposition,
\newblock \emph{Ann. Univ. Sci. Budapest,  E\"otv\"os Sect. Math.} \textbf{58} (2015),  p.\,27--35.
\bibitem{Fo} \textsc{G. Folland,}
\newblock A course in abstract harmonic analysis, \emph{Studies in 
Advanced Mathematics} (1995), CRC Press.

\bibitem{Gl} \textsc{E. Glasner,}
\newblock Ergodic theory via joinings, \emph{Mathematical Surveys and Monographs} \textbf{101} (2003),
American Mathematical Society.

\bibitem{GLR}
\textsc{H.~Gr\"ochenig, V.~Losert, H.~Rindler,}
\newblock Uniform distribution in solvable groups,
\newblock Probability measures on groups VIII (Oberwolfach, 1985), 
\emph{Lecture Notes in Math.} \textbf{1210} (1986),  Springer,
p.\,97--107. 

\bibitem{Gui} \textsc{A.~Guichardet,} 
\newblock Produits tensoriels infinis et repr\'esentations des relations 
d'anti\-com\-mutation,
\newblock \emph{Ann. Sci. \'E.N.S.} {\bf 83} (1966), p.\,1--52.


\bibitem{HarpeVal} \textsc{P.~de la Harpe, A.~Valette,} 
\newblock La propri\'et\'e (T) de Kazhdan pour les groupes localement compacts 
(avec un appendice de Marc Burger),
\newblock \emph{Ast\'erisque} {\bf 175} (1989), Societ\'e Math\'ematique de France. 

\bibitem{Ha1}
\textsc{S.~Hartman,} 
\newblock The method of Grothendieck-Ramirez and weak topologies in C(T), 
\newblock \emph{Studia Math.} \textbf{44} (1972), p.\,181--197. 

\bibitem{Ha2}
\textsc{S.~Hartman,} 
\newblock On harmonic separation,
\newblock \emph{Colloq. Math.} \textbf{42} (1979), p.\,209--222. 

\bibitem{HMP}
\textsc{B.~Host, J.-F.~M\'ela, F.~Parreau,}
\newblock Non-singular transformations and spectral analysis of measures,
\newblock  \emph{Bull. Soc. Math. France} \textbf{119} (1991), p.\,33--90. 

\bibitem{HM}
\textsc{R.~Howe, C.~Moore,}
\newblock Asymptotic properties of unitary representations,
\newblock \emph{J. Funct. Anal.} \textbf{32} (1979), p.\,72--96. 



\bibitem{Ka1}
\textsc{R.~Kaufman,}
\newblock Remark on Fourier-Stieltjes transforms of continuous measures, 
\newblock \emph{Colloq. Math.} \textbf{22} (1971) p.\,279--280.

\bibitem{Ka2}
\textsc{R.~Kaufman,} 
\newblock Continuous measures and analytic sets, 
\newblock \emph{Colloq. Math.} \textbf{58} (1989), p.\,17--21. 

\bibitem{K} 
\textsc{D.~Kazhdan,}
\newblock Connection of the dual space of a group with the structure of 
its closed subgroups,
\newblock \emph{Func. Anal. Appl.}  {\bf 1} (1967), p.\,63--65.

\bibitem{Ke} 
\textsc{D.~Kerr, H.~Li,}
\newblock Ergodic Theory: Independence and dichotomies, 
\newblock \emph{Springer Monographs in Math.}, Springer (2016).

\bibitem{KuiNied}
\textsc{L.~Kuipers, H.~Niederreiter,}
\newblock Uniform distribution of sequences,
\emph{Pure and Applied Mathematics}, Wiley-Interscience (1974). 

\bibitem{LR}
\textsc{V.~Losert, H.~Rindler,}
\newblock Uniform distribution and the mean ergodic theorem,
\newblock \emph{Inv. Math.} \textbf{50} (1978), p.\,65--74.



\bibitem{DM}
\textsc{B.~De~Mathan,}
\newblock Numbers contravening a condition in density modulo 1,
\newblock \emph{Acta Math. Acad. Sci. Hungar.} \textbf{36} (1980), p.\,237--241.


\bibitem{Mo} \textsc{C.~Moore,}
\newblock Groups with finite dimensional irreducible representations,
\newblock \emph{Trans. Amer. Soc.} {\bf 166} (1972), p.\,401--410.


\bibitem{VN} 
\textsc{J.~von~Neumann,} 
\newblock On infinite direct products, 
\newblock \emph{Compositio Math.} {\bf 6} (1938), p.\,1--77.

\bibitem{Pal}
\textsc{T.~Palmer,} 
\newblock Classes on nonabelian, noncompact, locally compact groups, 
\newblock \emph{Rocky Mountain J. of Math.} {\bf 8} (1978), p.\,683--741.

\bibitem{Pet} \textsc{J.~Peterson,}
\newblock Lecture notes on ergodic theory, available at 
\verb=http://www.math.=
\par\noindent
\verb=vanderbilt.edu/peterson/teaching/Spring2011/ErgodicTheoryNotes.pdf=

\bibitem{Pol}
\textsc{A.~D.~Pollington,}
\newblock On the density of sequences $(n_k\xi)$,
\newblock \emph{Illinois J. Math}  \textbf{23} (1979), p.\,511--515.


\bibitem{T}
\textsc{M.~Taylor,}
\newblock Noncommutative harmonic analysis,
\newblock \emph{Mathematical Surveys and Monographs} \textbf{22} (1986), American Mathematical Society.

\bibitem{V1}
\textsc{W.~Veech,}
\newblock Some questions of uniform distribution,
\newblock \emph{Ann. of Math.}  \textbf{94} (1971), p.\,125--138. 

\bibitem{V2}
\textsc{W.~Veech,} 
\newblock Topological dynamics, 
\newblock \emph{Bull. Amer. Math. Soc.} \textbf{83} (1977), p.\,775--830. 


\bibitem{Z}
\textsc{R.~Zimmer,}
\newblock Ergodic theory and semisimple groups, 
\newblock \emph{Monographs in Mathematics} \textbf{81} (1984), Birkh\"auser Verlag.



}
 \end{thebibliography}
\end{document}